\documentclass[12pt]{article}
\usepackage{amsthm,amsfonts,amssymb,epsfig,graphics,amsmath,amsbsy,tikz}

\usepackage{comment}
\usepackage{mathtools}
\usepackage{enumerate}
\mathtoolsset{showonlyrefs=true}

\newtheorem{theorem}{Theorem}[section] 

\newtheorem{proposition}[theorem]{Proposition}
\newtheorem{remark}[theorem]{Remark}
\newtheorem{definition}[theorem]{Definition}

\usepackage{upquote,textcomp} 

\usepackage[colorlinks = true,
            linkcolor = blue,
            urlcolor  = blue,
            citecolor = blue,
            anchorcolor = blue]{hyperref} 

\numberwithin{equation}{section}

\usepackage{listings}
\definecolor{dkgreen}{rgb}{0,0.6,0}
\definecolor{gray}{rgb}{0.5,0.5,0.5}

\usepackage[toc,page]{appendix}

\title{A New Measure of Coarseness for Solutions to Cahn--Hilliard Equations
}

\author{Peter Howard
\thanks{Corresponding author;
    Department of Mathematics, 
    Texas A\&M University, 
    College Station, 
    TX 77843, USA, 
    phoward@tamu.edu}
\and
Adam Larios
\thanks{Department of Mathematics, 
        University of Nebraska--Lincoln,
        Lincoln, NE 68588-0130, USA, 
        alarios@unl.edu}
\and
Quyuan Lin
\thanks{School of Mathematical and Statistical Sciences,              
    Clemson University, 
    Clemson, SC 29634, USA, 
    quyuanl@clemson.edu}
}

\date{\today}

\usepackage{placeins} 

\usepackage{subcaption}

\usepackage{todonotes}

\begin{document}

\maketitle

\begin{abstract} 
We introduce a new measure of coarseness for characterizing
phase separation processes such as those described by 
Cahn--Hilliard equations. An advantage of our measure is that it 
remains consistent throughout the evolution, including for solutions
with no periodic structure. We use our measure to compare two previous 
models of coarsening dynamics with numerically generated dynamics,
providing the first direct check that we are aware of for the 
efficacy of these methods. 
\end{abstract}


\section{Introduction}\label{introduction}

For processes of phase separation such as those modeled by Cahn--Hilliard
equations, solutions are typically viewed as evolving from less coarse 
states (e.g., from nearly homogeneous mixtures) to more coarse states 
(e.g., to fully separated mixtures). The rates at which such processes evolve
are of critical importance both for understanding the nature of the processes
and for applications, and in order to study these rates it is necessary to 
have a consistent and effective measure of coarseness that can be employed 
throughout the phase separation process. Perhaps the earliest such measure was introduced by 
J. S. Langer in \cite{L71}. Noting that late-stage solutions to 1D Cahn--Hilliard
equations are nearly periodic, Langer suggested using the approximate period, 
appropriately defined, as a measure of coarseness. Moreover, by combining this 
measure of coarseness with an approach to evolution based on perturbations 
instigated by thermal fluctuation in the materials under consideration, 
Langer was able to arrive at an elegant
analytical formula for the time evolution of coarseness in solutions to Cahn--Hilliard
equations. 

Motivated by Langer's work, the first author of the current study has introduced
an alternative analytical approach to coarsening dynamics based on the spectrum 
associated with certain unstable periodic stationary solutions to Cahn--Hilliard 
equations \cite{H11}. By utilizing exact periodic solutions rather than asymptotically 
approximate periodic solutions, this latter approach extends Langer's method to 
earlier stages of the coarsening process, providing additional information about 
the transition from mid-stage dynamics to late-stage dynamics. Nonetheless, there 
remains a critical 
period of minimal coarseness during which neither method reasonably applies. Indeed,
since both methods rely on periodicity of some form in order for the measure of 
coarseness to be defined, it is reasonable to assert that neither method provides 
meaningful values during the period of minimal coarseness. 

An alternative measure of coarseness was introduced by Kohn and Otto in \cite{KO02} 
in the multidimensional case (Definition 1 from p.~383 of \cite{KO02}). Translated 
to the current one-dimensional setting, this measure can be expressed as follows. 

\begin{definition} \label{Kohn-Otto-Definition}
    For any function $\phi (x)$ periodic on $[-L, +L]$ with mean 
    value 0, the Kohn--Otto length is defined to be 
    \begin{equation*}
        \ell_{KO} 
        := \sup \Big{\{} \frac{1}{2L} \int_{-L}^{+L} \phi (x) \zeta (x) dx: \,
        \zeta \in C^1 ([-L, +L]) \,\, \textrm{periodic with}\, \sup_{x \in [-L,+L]} |\zeta' (x)| \le 1\Big{\}}.
    \end{equation*}
\end{definition}
The clear advantage of the Kohn--Otto length over Langer's measure 
is that it readily generalizes to the multidimensional setting. Nonetheless, 
it retains dependence on periodicity, and so is not suitable for effectively 
capturing early-stage dynamics.   

The primary goal of the current analysis is to introduce a new measure of 
coarseness in the 1D setting that agrees closely with Langer's when solutions 
are nearly 
periodic, but can be naturally extended throughout the entirety of the 
coarsening process. This new measure allows us to compare the analytic 
methods of \cite{L71} and \cite{H11} with computational results, and 
better gauge the regions of efficacy for these methods. More generally, 
our measure of coarseness opens the door for robust 
studies of coarsening rates not only for Cahn--Hilliard equations, but 
for systems as well, such as Cahn--Hilliard--Navier--Stokes couplings---an 
avenue the authors will explore in a subsequent study. 

The remainder of the paper is organized as follows. In Section \ref{ch-coarsening}, 
we collect 
observations about Cahn--Hilliard equations that will be important to our 
later study, and in Section \ref{MeasureofCoursening} we use the observations to introduce our 
measure of coarseness. In Section \ref{coarsening-rates-section}, we review and compare the coarsening
rates computed in \cite{L71} and \cite{H11}, and in Section \ref{sec-computational-results} we use
our new measure of coarseness to make comparisons with numerically 
generated rates. In this way, we are able to determine the timescales
on which the analytic approaches are most effectively. We conclude
with a discussion of future work, and relegate two technical proofs
to a short appendix.

\section{Cahn--Hilliard Equations} 
\label{ch-coarsening}

In this section, we set the stage for our analysis by reviewing some 
key properties of Cahn--Hilliard equations,  
\begin{equation} \label{ch}
\phi_t = 
\Big( M(\phi) (-\kappa \phi_{xx} + F' (\phi))_x \Big)_x,
\end{equation}
where $\kappa > 0$ and for this discussion we will make the 
following assumptions on $M$ and $F$. 

\medskip
\noindent
{\bf (A)} $M \in C^2 (\mathbb{R})$, and there exists a constant 
$m_0 > 0$ so that with $M (\phi) \ge m_0$ for all $\phi \in \mathbb{R}$ 
(i.e., non-degenerate mobility);
$F \in C^4 (\mathbb{R})$ has a double-well form: there exist real
numbers $\alpha_1 < \alpha_2 < \alpha_3 < \alpha_4 < \alpha_5$ so
that $F$ is strictly decreasing on $(-\infty, \alpha_1)$ and
$(\alpha_3, \alpha_5)$ and strictly increasing on $(\alpha_1, \alpha_3)$
and $(\alpha_5, +\infty)$, and additionally $F$ is concave up on
$(-\infty, \alpha_2) \cup (\alpha_4, +\infty)$ and concave down on
$(\alpha_2,\alpha_4)$.  

\medskip

\begin{remark} With its origins in the work of John W. Cahn 
and John E. Hilliard, especially \cite{Cahn1961, CH58}, equation \eqref{ch}
is now well established as a foundational model of phase separation dynamics. 
Although a general review of references on \eqref{ch} is far 
beyond the scope of this discussion, we mention that our approach 
and methods are closely related to the work on initiation of 
phase separation in \cite{Grant91, Grant93}, the analyses of periodic 
solutions in \cite{H09, H11}, the analyses of kink and antikink solutions in 
\cite{BKT1999, CCO2001, H07, OW2014}, and the analyses of coarsening
rates in \cite{H11, L71, watson2003coarsening}.
\end{remark}

We observe at the outset that for each $F$ satisfying Assumptions {\bf (A)}, 
there exists a unique pair of values $\phi_1$ and $\phi_2$ 
(the {\it binodal} values) so that
\begin{equation}
F'(\phi_1) = \frac{F(\phi_2) - F(\phi_1)}{\phi_2 - \phi_1} = F'(\phi_2)
\end{equation}
and such that the line passing through
$(\phi_1, F(\phi_1))$ and $(\phi_2, F(\phi_2))$ lies entirely on or 
below $F$. (See Figure \ref{F-figure}.)
Also, we note that for any linear function 
$G(\phi) = A \phi + B$ we can replace $F(\phi)$ in \eqref{ch} with 
$H(\phi) = F(\phi) - G(\phi)$ without changing the equation in 
any way.  If we choose 
\begin{equation*}
G(\phi) := \frac{F(\phi_2) - F(\phi_1)}{\phi_2 - \phi_1} (\phi - \phi_1) + F(\phi_1),
\end{equation*}
then $H(\phi)$ has local minima at the binodal values, 
with $H(\phi_1) = H(\phi_2) = 0$, and a local maximum at the 
unique value $\phi_h$ for which 
\begin{equation*}
F'(\phi_h) = \frac{F(\phi_2) - F(\phi_1)}{\phi_2 - \phi_1}
\quad \textrm{and} \quad F''(\phi_h) < 0.
\end{equation*}
Finally, upon replacing $\phi$ with $\phi+\phi_h$ 
we can shift $H$ so that the local maximum is located at 
$\phi_h = 0$. For the remainder of our analysis we will assume 
that these transformations have been carried out, and we 
will denote the resulting function $F$. The 
standard form that we will use for numerical computations 
and some specific analytical results is 
\begin{equation} \label{quarticF}
F(\phi) = \frac{1}{4} \alpha \phi^4 - \frac{1}{2} \beta \phi^2
+ \frac{1}{4} \frac{\beta^2}{\alpha}
= \frac{\alpha}{4} \left(\phi^2 - \frac{\beta}{\alpha}\right)^2,
\end{equation} 
which clearly satisfies our general assumptions for all 
$\alpha, \beta > 0$.

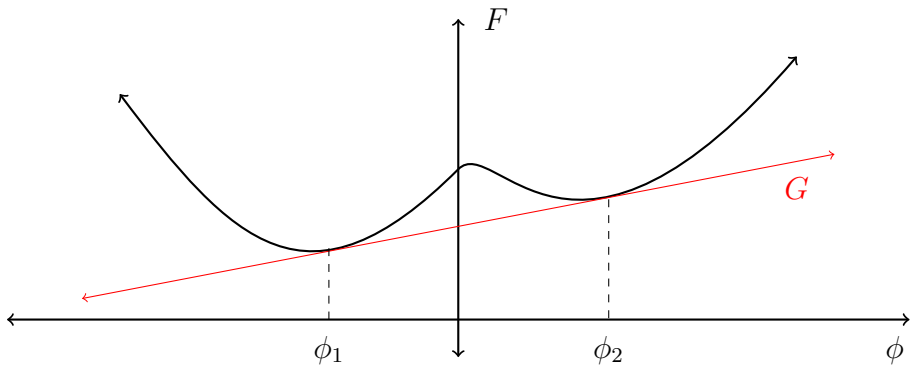
\begin{figure}[ht]
\begin{center}
\begin{tikzpicture}
\draw[thick, <->] (-6,0) -- (6,0);
\node at (5.8,-.4) {$\phi$};
\draw[thick,<->] (0,-.5) -- (0,4);
\node at (.5,4) {$F$};
\draw[thick,<-] (-4.5,3) .. controls (-3,1) and (-2,0) .. (0,2);
\draw[thick,->] (0,2) .. controls (.5,2.5) and (1.5,0) .. (4.5,3.5); 
\draw[<->,red] (-5,.28) -- (5,2.2);
\node at (4.5,1.75) {$\color{red} G$};
\draw[dashed] (-1.72,.95) -- (-1.72,0);
\node at (-1.72,-.4) {$\phi_1$};
\draw[dashed] (2,1.6) -- (2,0);
\node at (2,-.4) {$\phi_2$};
\end{tikzpicture}
\end{center}
\caption{The bulk free energy $F$ along with its supporting line.} 
\label{F-figure}
\end{figure}

\begin{remark} \label{F-remark}
For convenient reference, we summarize the properties that 
$F$ will have after the transformations described above:
$F \in C^4 (\mathbb{R})$, and there exist values 
$\phi_1 < \phi_3 < 0 < \phi_4 < \phi_2$ so that 
$F$ is strictly decreasing on $(-\infty, \phi_1)$ and
$(0, \phi_2)$ and strictly increasing on $(\phi_1, 0)$
and $(\phi_2, +\infty)$; $F$ is convex on
$(-\infty, \phi_3)$ and $(\phi_4, +\infty)$, and concave on
$(\phi_3,\phi_4)$; $F$ has local minima at $\phi_1$ and 
$\phi_2$ with $F(\phi_1) = F (\phi_2) = 0$, and $F$ has 
a local maximum at $\phi = 0$. 
\end{remark}

\subsection{The Cahn--Hilliard energy}

Equations of form \eqref{ch} are based on the energy
functional
\begin{equation} \label{ch_energy}
    E(\phi) := \int_{-L}^{+L} F (\phi) + \frac{\kappa}{2} \phi_x^2 dx,
\end{equation}
introduced for this context in \cite{CH58} (see also \cite{GNRW03, KO02}). 
In particular, \eqref{ch} can be expressed as a conservation law, 
\begin{equation*}
\phi_t + J_x = 0,
\end{equation*}
with flux 
\begin{equation*}
J = - M (\phi) \partial_x \frac{\delta E}{\delta \phi},
\end{equation*}
where $\frac{\delta E}{\delta \phi}$ denotes the usual variational 
gradient, and we have restricted the expressions from \cite{CH58} 
to one space dimension. A straightforward calculation shows that if $\phi$ evolves
according to \eqref{ch} then the energy $E (\phi)$ will generally 
dissipate as it approaches a global minimum value, corresponding with a 
stationary solution of \eqref{ch}. One important aspect of 
such solutions is the rate at which they coarsen from an 
initial near-homogeneous configuration to a distinctive
non-homogeneous configuration corresponding with an asymptotic
limit. In principle, we would like to gauge such coarsening by 
evolving some designated length scale $\ell (t)$, but in practice 
it is problematic to consistently assign such a scale to an arbitrary 
solution $\phi(x,t)$ of \eqref{ch}. Instead, it is often more convenient
to gauge coarsening by evolution of the
energy \eqref{ch_energy}, and one goal of the current analysis 
is to relate the evolution of $E(\phi (\cdot, t))$ to 
the evolution of a length scale $\ell (t)$ in a consistent manner. 

\subsection{Periodic solutions of Cahn--Hilliard equations}

For many phase separation processes, including the well-studied
process of spinodal decomposition, we expect 
$\phi (x,0) = \phi_0 (x)$ to be a small random perturbation of
a homogeneous state $\phi_h = \text{constant}$.  We can understand 
the initiation of phase separation by linearizing \eqref{ch} about 
this state ($\phi = \phi_h + v$) to obtain the linear perturbation equation
\begin{equation*} 
v_t = - M(\phi_h) \kappa v_{xxxx} + M(\phi_h) F''(\phi_h) v_{xx}.
\end{equation*}
If we look for solutions of the form $v(x,t) = e^{\lambda t + i\xi x}$ 
we obtain the dispersion relation 
\begin{equation*}
\lambda (\xi) = - \kappa M(\phi_h) \xi^4 - M(\phi_h) F''(\phi_h) \xi^2,
\end{equation*}
with leading eigenvalue 
\begin{equation} \label{Leaders}
\lambda_s = \frac{M(\phi_h) F''(\phi_h)^2}{4 \kappa} > 0; \quad
\text{at} \quad \xi_s = \sqrt{-\frac{F''(\phi_h)}{2 \kappa}},
\end{equation}
and associated period
\begin{equation} \label{spinodal-period}
p_s = 2\pi \sqrt{\frac{2\kappa}{- F''(\phi_h)}},    
\end{equation}
where we emphasize that under our assumptions on $F$,
$F''(\phi_h) < 0$.
Accordingly, we expect solutions of \eqref{ch}, initialized
by small random perturbations of $\phi_h$, to rapidly evolve 
toward a periodic solution with period $p_s$.  Indeed, for 
the case of \eqref{ch} posed on a bounded domain in $\mathbb{R}$,
this expectation has been rigorously verified by Grant 
\cite{Grant91,Grant93}. It is natural to view these initial 
dynamics as culminating once solutions are nearly periodic 
with period $p_s$, and we refer to the dynamics up to this 
time as the {\it spinodal phase} of the process. Correspondingly,
we refer to $p_s$ as the {\it spinodal period}. 

At the end of the spinodal phase, solutions will generally be near a stationary 
periodic solution with period $p_s$, and in order to understand
the next phase of the dynamics, we consider the behavior of 
solutions initialized as perturbations of stationary 
periodic solutions. As a starting point for this, we have 
from \cite{H09, H11} that there exists a continuum 
of such periodic solutions $\bar{\phi} (x)$ to \eqref{ch} in the 
following sense: if $\phi_1$ and $\phi_2$ denote the binodal values
and $\phi_{\min}$ and $\phi_{\max}$ are any values so that 
$\phi_1 < \phi_{\min} < \phi_{\max} < \phi_2$ with additionally
\begin{equation*}
F'(\phi_{\min}) > \frac{F(\phi_{\max}) - F(\phi_{\min})}{\phi_{\max} - \phi_{\min}} > F'(\phi_{\max}),
\end{equation*}  
then there exists a periodic solution to \eqref{ch}
with minimum value $\phi_{\min}$ and maximum value $\phi_{\max}$
(see Theorem 1.5 in \cite{H09}). Moreover (again from Theorem 
1.5 in \cite{H09}), we have that if 
$F$ is as described in Remark \ref{F-remark}, and if 
additionally $F$ is an even function (such as \eqref{quarticF}), 
then for every 
amplitude $a \in (0, \phi_2)$ there exists precisely one
(up to translation) periodic stationary solution with
amplitude $a$. More precisely, this solution, denoted here
$\bar{\phi} (x; a)$,  satisfies the relation
\begin{equation*}
- \kappa \bar{\phi}_{xx} + F'(\bar{\phi}) = 0 
\implies (\bar{\phi}_x)^2 = \frac{2}{\kappa} (F(\bar{\phi}) - F(a)).
\end{equation*}
If we select the shift so that $\bar{\phi} (0; a) = 0$, we find
the integral relation
\begin{equation} \label{integralu}
\int_0^{\bar{\phi} (x;a)} \frac{dy}{\sqrt{\frac{2}{\kappa} (F(y) - F(a))}} = x,
\end{equation}  
from which we see immediately (by symmetry) that $\bar{\phi}(x;a)$ 
has period
\begin{equation} \label{a-to-p}
p(a) = 4 \int_0^{a} \frac{dy}{\sqrt{\frac{2}{\kappa} (F(y) - F(a))}}.
\end{equation}

Returning briefly to the initiation of dynamics, we can use 
\eqref{a-to-p} to compute the minimum possible period, obtained
as the limit of $p(a)$ as $a$ tends to 0. For this calculation, 
we need only observe that for $a$ small and $y \in (0, a)$, 
we can Taylor expand $F(a)$ about $y$, and subsequently Taylor 
expand $F'(y)$ and $F''(y)$ about $0$ (and note
that $F'(0) = 0$) to write 
\begin{equation*}
    F(y) - F(a) = \frac{1}{2} F''(0) (y^2 - a^2) \Big(1 + \mathbf{O} (a)\Big).
\end{equation*}
We then have 
\begin{equation*}
    p (a) = 4 \sqrt{\frac{\kappa}{- F'' (0)}} \int_0^a \frac{1}{\sqrt{a^2 - y^2}} \Big(1 + \mathbf{O} (a) \Big) dy 
    = 2 \pi \sqrt{\frac{\kappa}{- F'' (0)}} + {\mathbf O} (a).
\end{equation*}
We see that the minimum period is 
\begin{equation} \label{minimim-period}
    p_{\min} := \lim_{a \to 0^+} p(a) 
    = 2 \pi \sqrt{\frac{\kappa}{- F'' (0)}}. 
\end{equation}

In the following proposition, we observe that the period $p(a)$
increases as the amplitude $a$ increases. In \cite{H11}, this 
result is stated without proof as equation (1.11),
and although the short proof is straightforward (as asserted in 
\cite{H11}), we have elected to include it in an appendix for 
completeness. (See Appendix \ref{proofs-appendix}.) 

\begin{proposition} \label{period-proposition}
Assume $F$ is as described in Remark \ref{F-remark}, and also 
that $F$ is an even function. Then for all $a \in (0, \phi_2)$ 
the period $p(a)$ specified in \eqref{a-to-p} satisfies 
\begin{equation} \label{period-derivative}
p'(a) = \frac{2 \sqrt{2 \kappa}}{\sqrt{F(0)-F(a)}} 
- \sqrt{2 \kappa} \int_0^{a} \frac{F'(y) - F'(a)}{(F(y) - F(a))^{3/2}} dy. 
\end{equation}
In addition, if $F'''(\phi) > 0$ for all $\phi \in (0, \phi_2)$ 
then $p'(a) > 0$ for all $a \in (0, \phi_2)$. 
\end{proposition}

Since $p$ is monotonically increasing as a function of amplitude, we can 
uniquely specify the spinodal amplitude $a_s$ so that 
$p(a_s) = p_s$. Precisely, we obtain the relation 
\begin{equation*}
\frac{2 \pi \sqrt{2 \kappa}}{\sqrt{- F''(0)}}
= 4 \sqrt{\kappa} \int_0^{a_s} \frac{dy}{\sqrt{2 (F(y) - F(a_s))}},
\end{equation*}
and since $\sqrt{\kappa}$ can be divided out of both sides, 
we see that $a_s$ does not depend on $\kappa$.

Using \eqref{integralu}, we can identify the unique (up to shift, 
selected by $\bar{\phi} (0;a_s) = 0$)
periodic solution of \eqref{ch} with amplitude $a_s$, namely
$\bar{\phi} (x; a_s)$, and subsequently we define the spinodal 
energy $E_s$ as the energy associated with this periodic solution,
\begin{equation} \label{spinodal-energy}
    E_s := E(\bar{\phi} (\cdot; a_s)).
\end{equation} 
For the long-time models discussed below, we will typically 
think of initiating the dynamics once the energy has 
reduced to the spinodal value. 

Although the energy function $E (\phi (\cdot, t))$ can 
achieve any value attainable via functions $\phi (x, t)$
in its domain, the dynamics we have in mind involve 
solutions with energies less than the energy achieved 
by the homogeneous configuration $\phi_{0} \equiv 0$.
For a given interval $[-L, +L]$, this is easily computed 
to be 
\begin{equation} \label{maximum-energy}
    E_{\max} := E (0) = \int_{-L}^{+L} F(0) dx
    = 2LF(0). 
\end{equation}
As time increases, energies will decrease as solutions 
approach either a kink or anti-kink solution, respectively
$K(x)$ or $K(-x)$, where 
\begin{equation*}
    - \kappa K'' + F' (K) = 0,
    \quad \forall\, x \in \mathbb{R},
\end{equation*}
with also 
\begin{equation*}
    \lim_{x \to -\infty} K(x) = \phi_1,
    \quad    \lim_{x \to +\infty} K(x) = \phi_2.
\end{equation*}
(See \cite{H09} for existence of such solutions, and 
\cite{H07} for asymptotic stability; we note that according 
to convention (see, e.g., \cite{EB1996}), 
a {\it kink} solution is monotonically 
increasing, while an {\it anti-kink} solution is 
monotonically decreasing.) These solutions
provide us with a lower bound on the 
energy 
\begin{equation}
    E_{\min} := \int_{-L}^{+L}
    F(K (x)) + \frac{\kappa}{2} |K'(x)|^2 dx,
\end{equation}
which will be computed explicitly in Proposition 
\ref{specific-F-proposition} just below. 

For specific implementations of our approach, we will use 
the family of bulk free energy densities specified in 
\eqref{quarticF}. Many of the preceding relations can 
be made explicit in this case, and for convenient reference
we summarize these in Proposition \ref{specific-F-proposition}
just below. Several of the statements in 
this proposition are taken or adapted from the references 
\cite{H07,H09,H11}, and others are straightforward 
calculations, so the short proof is relegated to 
Appendix \ref{proofs-appendix}.

\begin{proposition} \label{specific-F-proposition}
    For \eqref{ch}, let $M$ satisfy the assumptions in 
    {\bf (A)}, and let $F$ be as in \eqref{quarticF}.
    Then the following hold:

    \medskip
    \noindent
    (i) {\bf Minimum period}. The minimum period computed in \eqref{minimim-period} 
    is $p_{\min} = 2 \pi \sqrt{\kappa/\beta}$, and the 
    associated energy computed in \eqref{maximum-energy}
    is $E_{\max} = L\beta^2/(2 \alpha)$. 

    \medskip
    \noindent
    (ii) {\bf Spinodal period}. The spinodal period computed 
    in \eqref{spinodal-period} is $p_s = 2 \pi \sqrt{2 \kappa/\beta}$.

    \medskip
    \noindent
    (iii) {\bf Periodic solutions}. For each $a \in (0, \sqrt{\beta/\alpha})$,
    the periodic solution $\bar{\phi} (x; a)$ specified in 
    \eqref{integralu} can be expressed as a Jacobi elliptic function 
    \begin{equation} \label{jacobiellipticu}
    \bar{\phi} (x; a) = a \operatorname{sn} \left(\sqrt{\frac{-2(F(a) - F(0))}{\kappa}} \frac{x}{a}, k\right),
    \end{equation}
    where 
    \begin{equation} \label{kdefined}
    k = \sqrt{- \frac{\alpha a^4}{4 (F(a) - F(0))}}.
    \end{equation}
  
    \medskip
    \noindent
    (iv) {\bf Kink solutions and the minimum energy}. The binodal values for $F$ are $\pm \sqrt{\beta/\alpha}$, 
    and there exists a unique (up to translation) kink solution
    of \eqref{ch}, 
    \begin{equation} \label{kink-solution}
        K(x) = \sqrt{\frac{\beta}{\alpha}} \tanh \Big(\sqrt{\frac{\beta}{2 \kappa}} x \Big).
    \end{equation}
    The energy associated with this kink solution on $[-L, +L]$ 
    is 
    \begin{equation}\label{E_kink}
        E_{\min} := E(K(x)) = \int_{-L}^{+L} F (K (x)) + \frac{\kappa}{2} K'(x)^2 dx
        = \sqrt{2 \kappa \alpha} K (L) \Big(\frac{\beta}{\alpha} - \frac{K (L)^2}{3} \Big),   
    \end{equation}
    and it follows that 
    \begin{equation*}
     E_{\min}^{\infty} := \lim_{L \to \infty} E_{\min}
     = \frac{2}{3} \frac{\beta^2}{\alpha} \sqrt{\frac{2 \kappa}{\beta}}.   
    \end{equation*}
\end{proposition}

\begin{remark} \label{periodic-solution-details}
For the periodic solution \eqref{jacobiellipticu}, 
$\operatorname{sn} (y;k)$ denotes the Jacobi elliptic function, 
defined so that 
\begin{equation*}
\operatorname{sn} (y;k) = \sin \theta; \quad \text{where} \quad
y = \int_0^{\theta} \frac{d \zeta}{\sqrt{1 - k^2 \sin^2 \zeta}}.
\end{equation*}
For implementations, we evaluate $\operatorname{sn} (y; k)$ with 
the MATLAB function {\it ellipj.m}.
%
%
The period in this case is 
\begin{equation} \label{jacobiperiod}
p (a) = \frac{4 a \sqrt{\kappa}}{\sqrt{-2(F(a) - F(0))}} \mathcal{K} (k),
\end{equation}
where $\mathcal{K}$ denotes the complete elliptic integral
\begin{equation*}
\mathcal{K}(k) = \int_0^1 \frac{ds}{\sqrt{(1-s^2)(1-k^2 s^2)}}.
\end{equation*}
For implementations, we evaluate $\mathcal{K} (k)$ with 
the MATLAB function {\it ellipke.m}. 
\end{remark}

\begin{remark}
    \textcolor{black}{
    We observe for consistency that the expression in (iii) for $\bar{\phi} (x; a)$, $a\in(0,\sqrt{\frac{\beta}{\alpha}})$, 
    matches the results when $a\to 0$ and $a\to \sqrt{\frac{\beta}{\alpha}}$. To see this, we first take $a\to 0$ and observe that
    \[
      \lim_{a\to0}\sqrt{\frac{-2(F(a)-F(0))}{a^2\kappa}} x = \sqrt{\frac\beta\kappa} x \quad \text{ and } \quad \lim_{a\to0} k = \lim_{a\to0} \sqrt{\frac{-\alpha a^4}{4(F(a)-F(0))}} = 0.
    \]
    As $\operatorname{sn}(u;0) = \sin u$, one can conclude that 
    \[
     \lim_{a\to0} \bar\phi(x;a) = \lim_{a\to0} a \sin (\sqrt{\frac\beta\kappa} x) =0.
    \]
    Next, taking $a\to \sqrt{\frac{\beta}{\alpha}}$ we find that
    \[
      \lim_{a\to \sqrt{\frac{\beta}{\alpha}}}\sqrt{\frac{-2(F(a)-F(0))}{a^2\kappa}} x = \sqrt{\frac{\beta}{2\kappa}} x \quad \text{and} \quad \lim_{a\to \sqrt{\frac{\beta}{\alpha}}} k = \lim_{a\to \sqrt{\frac{\beta}{\alpha}}} \sqrt{\frac{-\alpha a^4}{4(F(a)-F(0))}} = 1.
    \]
    As $\operatorname{sn}(u;1) = \tanh u$, we obtain
    \[
      \lim_{a\to \sqrt{\frac{\beta}{\alpha}}} \bar\phi(x;a) = \sqrt{\frac\beta\alpha} \operatorname{sn}(\sqrt{\frac{\beta}{2\kappa}} x,1) = \sqrt{\frac\beta\alpha} \tanh(\sqrt{\frac{\beta}{2\kappa}} x) = K(x).
    \]
    which gives the kink solution.
    }
\end{remark}

As a working example for quantitative analysis, we will take 
$M (\phi) \equiv 1$ and use $F$ as in \eqref{quarticF} with 
$\alpha = \beta = 1$, leaving only $\kappa$ to vary. The choice of constant mobility is made  because qualitative dynamics do not vary much as long as the mobility is
non-degenerate. As a 
baseline case, we will take $\kappa = 0.001$. Using these values,  
periodic solutions
with three different periods are depicted in Figure 
\ref{pwaves-figure}. 

\begin{figure}[ht] 
\begin{center}
\includegraphics[width=12cm,height=8.2cm]{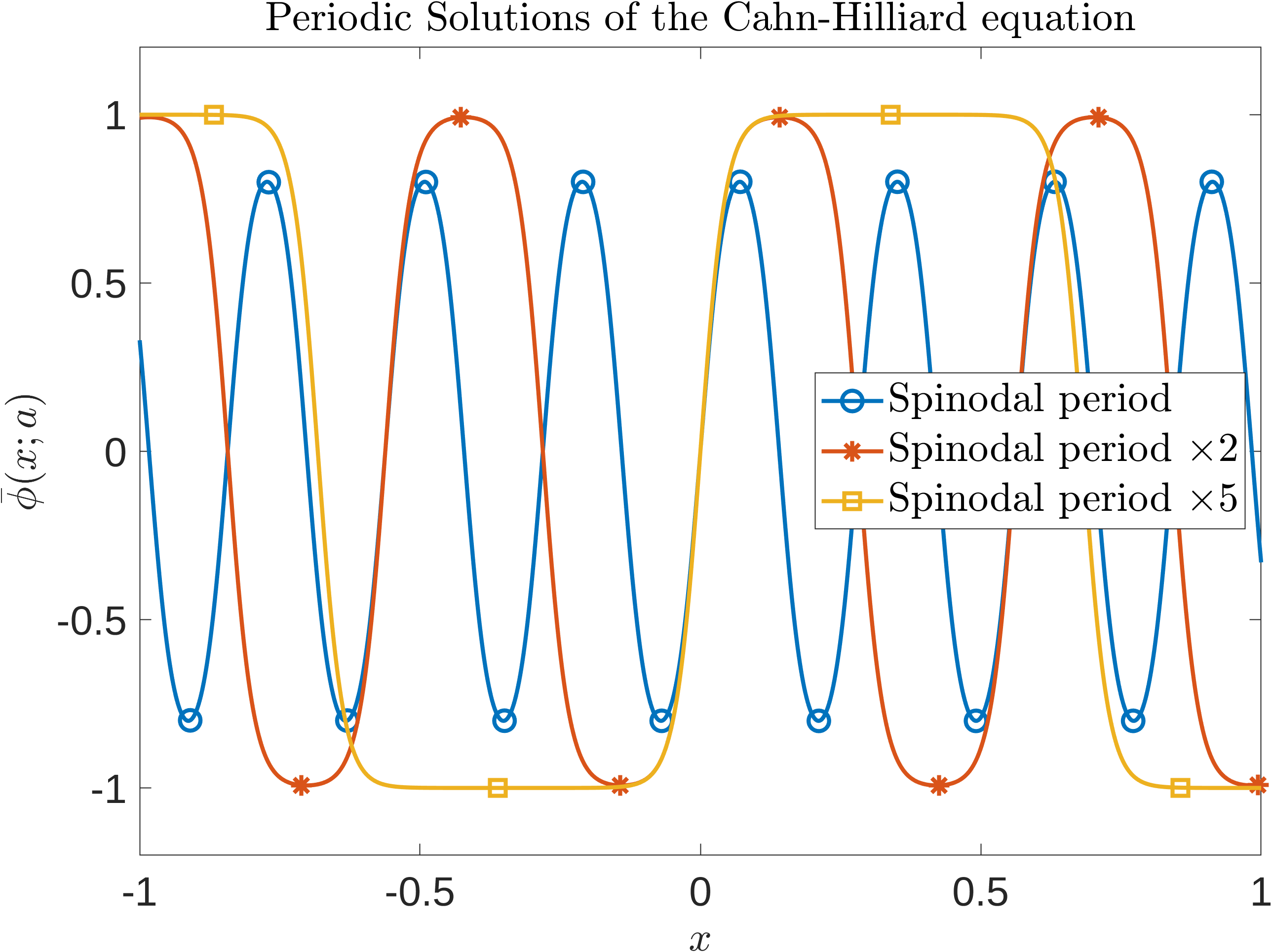}
\end{center}
\caption{Periodic solutions for the Cahn--Hilliard equation. Note that the domain $[-1,1]$ is depicted for visual clarity, but by Proposition~\ref{specific-F-proposition} (ii, the periods $p_s$ are approximately $0.2810$,  $0.5620$, and $1.4050$. \label{pwaves-figure}}
\end{figure}

\section{A New Measure of Coarseness} 
\label{MeasureofCoursening}

We return now to our goal set out previously of relating 
each energy $e \in (E_{\min}, E_{\max})$ with an associated 
length scale $\ell$. Since solutions to \eqref{ch} tend to 
be near periodic solutions during a substantial portion of the 
evolution, our strategy will be to identify 
with a given energy $e$ the period $p$ for which 
\begin{equation*}
    E (\bar{\phi} (\cdot; a (p))) = e. 
\end{equation*}
One caveat for this approach is that the function 
\begin{equation} \label{p-to-e}
    \mathcal{E} (p) 
    := E (\bar{\phi} (\cdot; a (p)))
    = \int_{-L}^{+L} F (\bar{\phi} (x; a(p))) 
    + \frac{\kappa}{2} \bar{\phi}_x (x; a(p))^2 dx
\end{equation}
is not invertible, a point we clarify in the next proposition. 

\begin{proposition} \label{EofPNotInvertible}
For \eqref{ch}, let $M$ satisfy the assumptions in 
{\bf (A)}, and let $F$ be as in \eqref{quarticF}.
Then the function $\mathcal{E} (p)$
defined in \eqref{p-to-e} is not invertible on $[p_{\min}, \infty)$. 
\end{proposition}

\begin{proof}
    First, since the period $p$ depends monotonically on the 
    amplitude $a$, we can work with the map 
\begin{equation} \label{energy-period}
    \mathbb{E} (a) := E (\bar{\phi} (\cdot; a))
    = \int_{-L}^{+L} F (\bar{\phi} (x; a)) + \frac{\kappa}{2} \bar{\phi}_x (x; a)^2 dx.
\end{equation}
We will prove the proposition by showing that $\mathbb{E}' (a)$ does not 
have a fixed sign. Upon differentiating $\mathbb{E}$ in $a$, we 
obtain 
\begin{equation} \label{energy-period-derivative}
\begin{aligned}
    \frac{d\mathbb{E}}{da} (a) 
    &= \int_{-L}^{+L} F' (\bar{\phi} (x; a)) \bar{\phi}_a (x; a) + \kappa \bar{\phi}_x (x; a) \bar{\phi}_{x a} (x; a) dx \\
    &\overset{\textrm{parts}}{=} \int_{-L}^{+L} F' (\bar{\phi} (x; a)) \bar{\phi}_a (x; a) - \kappa \bar{\phi}_{xx} (x; a) \bar{\phi}_{a} (x; a) dx \\
    &\quad \quad + \kappa \bar{\phi}_x (L; a) \bar{\phi}_{a} (L; a) - \kappa \bar{\phi}_x (-L; a) \bar{\phi}_{a} (-L; a) \\ 
    &= \kappa \bar{\phi}_x (L; a) \bar{\phi}_{a} (L; a) - \kappa \bar{\phi}_x (-L; a) \bar{\phi}_{a} (-L; a),
\end{aligned}    
\end{equation}
where the final equality follows because $- \kappa \bar{\phi}_{xx} (x; a) + F' (\bar{\phi} (x; a)) = 0$ for a periodic 
wave $\bar{\phi} (x; a)$. 

In order to understand the derivative $\bar{\phi}_a (x; a)$, we recall the explicit form for 
$\bar{\phi} (x; a)$, 
\begin{equation} \label{periodic-wave}
    \bar{\phi} (x; a) = a \operatorname{sn} (h(a)x; k), 
    \quad h(a) = \frac{1}{a} \sqrt{\frac{-2 (F(a) - F(0))}{\kappa}}.
\end{equation}
With $F$ as specified in \eqref{quarticF}, we find that 
\begin{equation*}
    F(a) - F(0) = a^2 (\frac{\alpha}{4} a^2 - \frac{\beta}{2}),
\end{equation*}
so 
\begin{equation*}
    h(a) = \sqrt{\frac{\alpha}{2\kappa}} \sqrt{\frac{2\beta}{\alpha} - a^2},
    \quad h'(a) = -\frac{\sqrt{\frac{\alpha}{2\kappa}}a}{\sqrt{\frac{2\beta}{\alpha} - a^2}}.
\end{equation*}
Now, 
\begin{equation*}
    \bar{\phi}_a (x; a) = \operatorname{sn} (h(a)x; k) + a \operatorname{sn}' (h(a)x; k) h'(a) x,
\end{equation*}
and likewise 
\begin{equation*}
    \bar{\phi}_x (x; a) = a \operatorname{sn}' (h(a)x; k) h(a),
\end{equation*}
from which we can write 
\begin{equation*}
    \bar{\phi}_a (x; a) = \frac{1}{a} \bar{\phi} (x; a) 
    + \frac{h'(a)x}{h(a)} \bar{\phi}_x (x; a)
    = \frac{1}{a} \bar{\phi} (x; a) - \frac{ax}{\frac{2\beta}{\alpha} - a^2} \bar{\phi}_x (x;a).
\end{equation*}

From these considerations, we see that 
\begin{equation*}
    \bar{\phi}_x (x; a) \bar{\phi}_a (x; a)
    = \frac{1}{a} \bar{\phi} (x; a) \bar{\phi}_x (x; a) 
    - \frac{ax}{\frac{2\beta}{\alpha} - a^2} \bar{\phi}_x (x;a)^2.
\end{equation*}
Since $\bar{\phi} (x; a)$ is an odd function (because $\operatorname{sn} (x;k)$ is odd), we 
see that the product $\bar{\phi}_x (x; a) \bar{\phi}_a (x; a)$ is an odd function, 
and from this we can conclude from \eqref{energy-period-derivative} that 
\begin{align} \label{energy-amplitude-derivative2}
    \frac{d\mathbb{E}}{da} (a) 
    & = 2 \kappa \bar{\phi}_x (L; a) \bar{\phi}_{a} (L; a)
    = \frac{2 \kappa}{a} \bar{\phi} (L; a) \bar{\phi}_x (L; a) 
    - \frac{2 \kappa aL}{\frac{2\beta}{\alpha} - a^2} \bar{\phi}_x (L;a)^2 \nonumber\\
    &= \frac{2 \kappa}{a} \bar{\phi}_x (L; a) \Big( \bar{\phi} (L; a) 
    - \frac{a^2L}{\frac{2\beta}{\alpha} - a^2} \bar{\phi}_x (L;a)\Big).
\end{align}
Recalling that our goal is to understand the sign of $\mathbb{E}' (a)$,
as $a$ varies, we begin by considering a configuration in which $\bar{\phi} (L; a) = 0$.
In this case, 
\begin{align*}
     \frac{d\mathbb{E}}{da} (a) 
     &= - \frac{2 \kappa aL}{\frac{2\beta}{\alpha} - a^2} \bar{\phi}_x (L;a)^2
     = - \frac{2 \kappa aL}{\frac{2\beta}{\alpha} - a^2} \frac{2}{\kappa}(F (0) - F(a)) \\
     &= - \frac{4 aL}{\frac{2\beta}{\alpha} - a^2} \frac{\alpha a^2}{4} \left(\frac{2\beta}{\alpha} - a^2\right)
     = - a^3 \alpha L < 0.
\end{align*}

With this configuration, we have either $\bar{\phi}_x (L; a) < 0$ or $\bar{\phi}_x (L; a) > 0$,
and for specificity we will focus on the former case. If we now think about increasing $a$ a small
amount, we know from the monotonic dependence of $p$ on $a$ that the period $p$ will increase 
a small amount and we will have $\bar{\phi} (L; a) > 0$. As the amplitude $a$ continues to 
increase the value of $\bar{\phi} (L; a)$ will increase, and correspondingly the value of 
$|\bar{\phi}_x (L; a)|$ will decrease. Nonetheless, for $\bar{\phi}_x (L; a) < 0$ and 
$\bar{\phi}(L;a) > 0$ we will continue to have $\mathbb{E}' (a) < 0$ until $\bar{\phi} (L;a)$
achieves its maximum value $\bar{\phi} (L;a) = a$, at which point $\mathbb{E}' (a) = 0$ (because 
$\bar{\phi}_x (L;a) = 0$).
As $a$ increases further, we will have $\bar{\phi}_x (L;a) > 0$, with $|\bar{\phi}_x (L; a)|$ 
sufficiently small so that 
\begin{equation*}
    \bar{\phi} (L; a) 
    - \frac{a^2L}{\frac{2\beta}{\alpha} - a^2} \bar{\phi}_x (L;a) > 0.
\end{equation*}
Since we are now in the setting with $\bar{\phi}_x (L; a) > 0$, we see from 
\eqref{energy-amplitude-derivative2} that we will have $\mathbb{E}' (a) > 0$. In this way, 
we conclude that $\mathbb{E} (a)$ is not monotonically decreasing as $a$ increases. 
Due to the monotonic dependence of $a$ on $p$, we can additionally conclude that 
the energy map $\mathcal{E} (p)$ is not monotonic in $p$, and so is not invertible 
for all values of $p$. 
\end{proof}

According to Proposition \ref{EofPNotInvertible}, we cannot invert 
$\mathcal{E} (p)$, but we will see below that the pseudoinverse
of $\mathcal{E} (p)$ nonetheless provides an effective measure of 
coarsening. 

\begin{definition} \label{pseudoinverse}
    Given a fixed energy $e \in (E_{\min}, E_{\max}]$, we define the associated
    period $p$ to be the pseudoinverse 
    \begin{equation} \label{pseudoinverse_eq}
        p = \inf \{\tilde p>0: \mathcal{E} (\tilde p) \le e\}.
    \end{equation}
\end{definition}

To better understand why the pseudoinverse works well in this case, 
we observe that $\mathcal{E} (p)$ is a mostly 
decreasing function of $p$ except on some small intervals on which $\mathcal{E}' (p)$ is close
to 0 (see, e.g., Figure \ref{EofPzoomed}). In Figure \ref{energy-period-figure}, we plot $\mathcal{E} (p)$
in the case $\alpha = \beta = 1$, $\kappa = 0.001$, and $L = 1$. We see
from this plot that as $p$ increases, sharp gradients with $\mathcal{E}' (p) < 0$
alternate with plateaus where $\mathcal{E}' (p) \cong 0$. The sharp 
gradients occur near values of $p$ for which $\bar{\phi} (L;a (p)) = 0$,
as in the proof of Proposition \ref{EofPNotInvertible}. Proposition
\ref{EofPPlateaus} addresses the nature of the plateaus. 

\begin{figure}[ht] 
\begin{center}
\includegraphics[width=12cm,height=8.2cm]{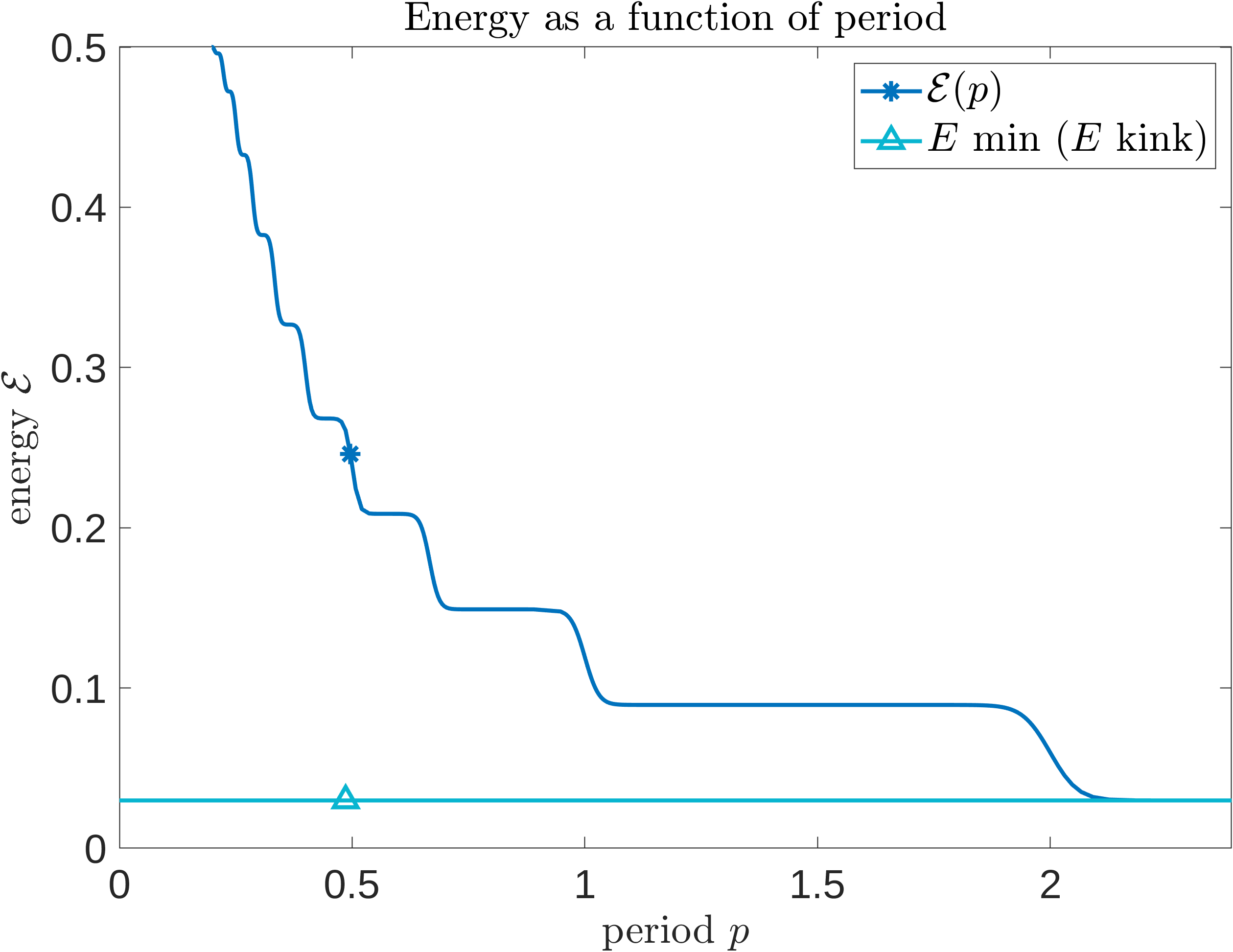}
\end{center}
\caption{Plot of $\mathcal{E} (p)$, computed with $\alpha = \beta = 1$, $\kappa = 0.001$, and $L = 1$. 
\label{energy-period-figure}}
\end{figure}

\begin{proposition} \label{EofPPlateaus}
For \eqref{ch}, let $M$ satisfy the assumptions in {\bf (A)}, and let $F$ be as in \eqref{quarticF}.
Then for any $p > p_{\min}$ for which $\mathcal{E}' (p) > 0$, we have the inequality 
\begin{equation*}
    \mathcal{E}' (p) 
    \le \frac{a^3 \sqrt{\alpha \kappa} (1+\delta)^{3/2} \delta^3 (\delta - 1)}{2L},
\end{equation*}
where $\delta = \sqrt{2\beta/(\alpha a^2)-1}$ and $a$ is the unique amplitude corresponding
with $p$. 
\end{proposition}

\begin{proof}
Recalling \eqref{energy-amplitude-derivative2} from the proof of Proposition 
\ref{EofPNotInvertible}, we see that if $\bar{\phi}_x (L;a) > 0$ then 
we can only have $\mathbb{E}' (a) > 0$ if 
\begin{equation} \label{barphiInequality}
    \frac{a^2L}{\frac{2\beta}{\alpha} - a^2} \bar{\phi}_x (L;a) < \bar{\phi} (L; a), 
\end{equation}
so that 
\begin{equation*}
    \bar{\phi}_x (L;a)
    \le \frac{\frac{2\beta}{\alpha} - a^2}{a^2L} \bar{\phi} (L; a). 
\end{equation*}
If we substitute this inequality into \eqref{energy-amplitude-derivative2}, we
see that 
\begin{equation} \label{EprimeEstimate}
    \frac{d \mathbb{E}}{da} (a) 
    \le \frac{2 \kappa}{a} \bar{\phi}_x (L; a) \bar{\phi} (L; a)
    < \frac{2 \kappa}{a} \frac{\frac{2\beta}{\alpha} - a^2}{a^2L} \bar{\phi} (L; a)^2
    \le \frac{2 \kappa}{a} \frac{\frac{2\beta}{\alpha} - a^2}{L},
\end{equation}
where in obtaining the final inequality we have observed that $|\bar{\phi} (L;a)| \le a$.

In order to translate this into information about the size of $\mathcal{E}' (p)$, we 
observe the relation 
\begin{equation} \label{ChainRule}
    \mathcal{E}' (p) = \mathbb{E}'(a)/p'(a).
\end{equation}
From the proof of Proposition \ref{period-proposition} (see Appendix \ref{proofs-appendix}),
we know that 
\begin{equation} \label{pprime}
    p' (a) = 2\sqrt{2\kappa} \int_0^1 \frac{G(az) - G(a)}{(F(az) - F(a))^{3/2}} \,dz,
\end{equation}
where 
\begin{equation*}
    G (y) = F (y) - \frac{y}{2} F'(y),
\end{equation*}
and moreover noted in the same proof that for all $a \in (0, \sqrt{\beta/\alpha})$ we 
have $p' (a) > 0$ (under the assumptions of that proposition). For our choice 
of $F$ \eqref{quarticF}, we have the relations 
\begin{equation*}
    \begin{aligned}
        F(az) - F(a) &= \frac{\alpha a^4}{4} (1 - z^2) (\delta^2 - z^2), \\
        G(az) - G(a) &= \frac{\alpha a^4}{4} (1-z^2) (1+z^2),
    \end{aligned}
\end{equation*}
where for notational brevity we have set 
\begin{equation*}
    \delta^2 := \frac{2 \beta}{\alpha a^2} - 1,
\end{equation*}
which is greater than 1 for all $a \in (0, \sqrt{\frac{\beta}{\alpha}})$. Upon 
substituting these relations into \eqref{pprime}, we arrive at the 
relation 
\begin{equation*}
    \begin{aligned}
    p' (a) &= \frac{4\sqrt{2 \kappa}}{a^2 \sqrt{\alpha}} 
    \int_0^1 \frac{1+z^2}{(1-z^2)^{1/2} (\delta^2 - z^2)^{3/2}} \,dz \\
    &=  \frac{4\sqrt{2 \kappa}}{a^2 \sqrt{\alpha}} 
     \int_0^1 \frac{1+z^2}{[(1-z)(1+z)]^{1/2} [(\delta - z)(\delta + z)]^{3/2}} \,dz. 
    \end{aligned}
\end{equation*}
We are primarily interested in a lower bound on this quantity, 
\begin{equation*}
    \begin{aligned}
    p'(a) &\ge \frac{4\sqrt{2 \kappa}}{a^2 \sqrt{\alpha}} 
    \frac{1}{\sqrt{2} (1+\delta)^{3/2}}
    \int_0^1 \frac{1}{(1-z)^{1/2} (\delta - z)^{3/2}} \,dz \\
    &\ge \frac{4\sqrt{\kappa}}{a^2 \sqrt{\alpha} (1+\delta)^{3/2}} 
    \int_0^1 \frac{1}{(\delta - z)^{2}} \,dz, 
    \end{aligned}
\end{equation*}
where the second inequality follows immediately from the observation that $\delta > 1$. 
Integrating directly, we now arrive at the inequality 
\begin{equation} \label{pprimeEstimate}
    p' (a) \ge \frac{4\sqrt{\kappa}}{a^2 \sqrt{\alpha} (1+\delta)^{3/2}} 
    \Big(\frac{1}{\delta (\delta - 1)} \Big).
\end{equation}
The stated estimate on $\mathcal{E}' (p)$ now follows from \eqref{ChainRule}, 
\eqref{EprimeEstimate}, and \eqref{pprimeEstimate}. 

Finally, we recall that these calculations have been carried out 
under the assumption $\bar{\phi}_x (L;a) > 0$, and note that the case 
with $\bar{\phi}_x (L;a) < 0$ is similar.
\end{proof}

\begin{remark}
As expected, we see from Proposition \ref{EofPPlateaus} that as the amplitude $a$ tends toward the 
maximum amplitude $\sqrt{\beta/\alpha}$, $p'(a)$ tends toward $\infty$ (since $\delta$ 
tends toward 1). I.e., for late-stage coarsening, $\mathcal{E}' (p)$ is much smaller 
than $\mathbb{E}' (a)$. For intermediate values of $a$, $p' (a)$ is proportional to 
$\sqrt{\kappa}$, so that $\mathcal{E}' (p)$ is (also) proportional to $\sqrt{\kappa}$. 

Although we are primarily interested in intermediate- to late-stage coarsening, it is perhaps
interesting to note the behavior of $\mathcal{E}' (p)$ as $a$ tends toward 0 (so $p$ tends toward
$p_{\min}$). As noted previously, in the event that $\mathbb{E}' (a) > 0$, we have the inequality 
\begin{equation*}
    \mathbb{E}' (a) < \frac{2\kappa}{a} \bar{\phi}_x (L; a) \bar{\phi} (L; a).
\end{equation*}
Recalling \eqref{periodic-wave}, we see that 
\begin{equation*}
    |\bar{\phi}_x (L; a)| 
    \le a \sqrt{\frac{\beta}{\kappa}},
\end{equation*}
where we have noted that $|\operatorname{sn}' (x; k)| \le 1$ for all $x \in \mathbb{R}$
(see, e.g., \cite{NIST:DLMF}). 
In this way, we conclude that 
\begin{equation*}
    \mathbb{E}' (a) \le 2a \sqrt{\beta \kappa}.
\end{equation*}
We can also obtain an alternative lower bound on $p' (a)$,
\begin{equation*}
    \begin{aligned}
        p' (a) &= \frac{4\sqrt{2 \kappa}}{a^2 \sqrt{\alpha}} 
     \int_0^1 \frac{1+z^2}{(1-z^2)^{1/2} (\delta^2 - z^2)^{3/2}} dz \\
     &\ge \frac{4\sqrt{2 \kappa}}{a^2 \delta^3 \sqrt{\alpha}} \int_0^1 \frac{1}{\sqrt{1-z^2}} dz
     = \frac{2 \pi\sqrt{2 \kappa}}{a^2 \delta^3 \sqrt{\alpha}}.
    \end{aligned}
\end{equation*}
Combining these observations, we see that 
\begin{equation*}
    \mathcal{E}' (p) = \mathbb{E}' (a)/p'(a)
    \le 2a \sqrt{\beta \kappa} \frac{a^2 \delta^3 \sqrt{\alpha}}{2 \pi\sqrt{2 \kappa}}
    = \frac{\sqrt{\alpha \beta}}{\pi \sqrt{2}}\left(\frac{2\beta}{\alpha} - a^2\right)^{3/2}.
\end{equation*}
We conclude that $\mathcal{E}' (p)$ is bounded above as $p$ approaches $p_{\min}$, though 
no longer proportional to $\sqrt{\kappa}$. 
\end{remark}

We now precisely specify our measure of coarseness. 

\begin{definition} For any solution $\phi(x,t)$ to \eqref{ch} with energy 
$E (\phi (\cdot,t)) \in (E_{\min}, E_{\max}]$, we define the coarseness $\ell$ to be the 
period $p$ defined in \eqref{pseudoinverse_eq} with $e = E (\phi (\cdot,t))$.
\end{definition}

In specifying this notion of coarseness, our focus has been on identifying 
a measure that can be used in a consistent way from the onset of spinodal decomposition through
the late stages of coarsening. The energy \eqref{ch_energy} is often used in precisely this way, 
but it has the disadvantage that it does not measure actual coarseness (i.e., a characteristic
length describing the state of the alloy) at any time during the dynamics. Our approach on the 
other hand, has the advantages of energy during early-stage dynamics, and is also a genuine measure 
of coarseness during the later stages of the dynamics. This allows us to study the rate of 
coarsening dynamics throughout the phase-separation process, and better understand changes that 
occur as the process unfolds. We will work with this measure of coarseness throughout the remainder 
of the analysis.

As an illustration of the considerations discussed in this section, we consider again 
the specific case of \eqref{ch} with $\kappa = 0.001$, $M (\phi) \equiv 1$, and $F$ as in \eqref{quarticF} 
with $\alpha = 1$ and $\beta = 1$. Referring to Figure \ref{energy-period-figure}, we pick the 
plateau between about 0.29 and 0.32, and 
in Figure \ref{EofPzoomed}
we zoom 
in on this flat section to see that in fact there is a small interval of periods for 
which monotonicity is lost. 


\begin{figure}[ht] 
\includegraphics[width=1\textwidth]{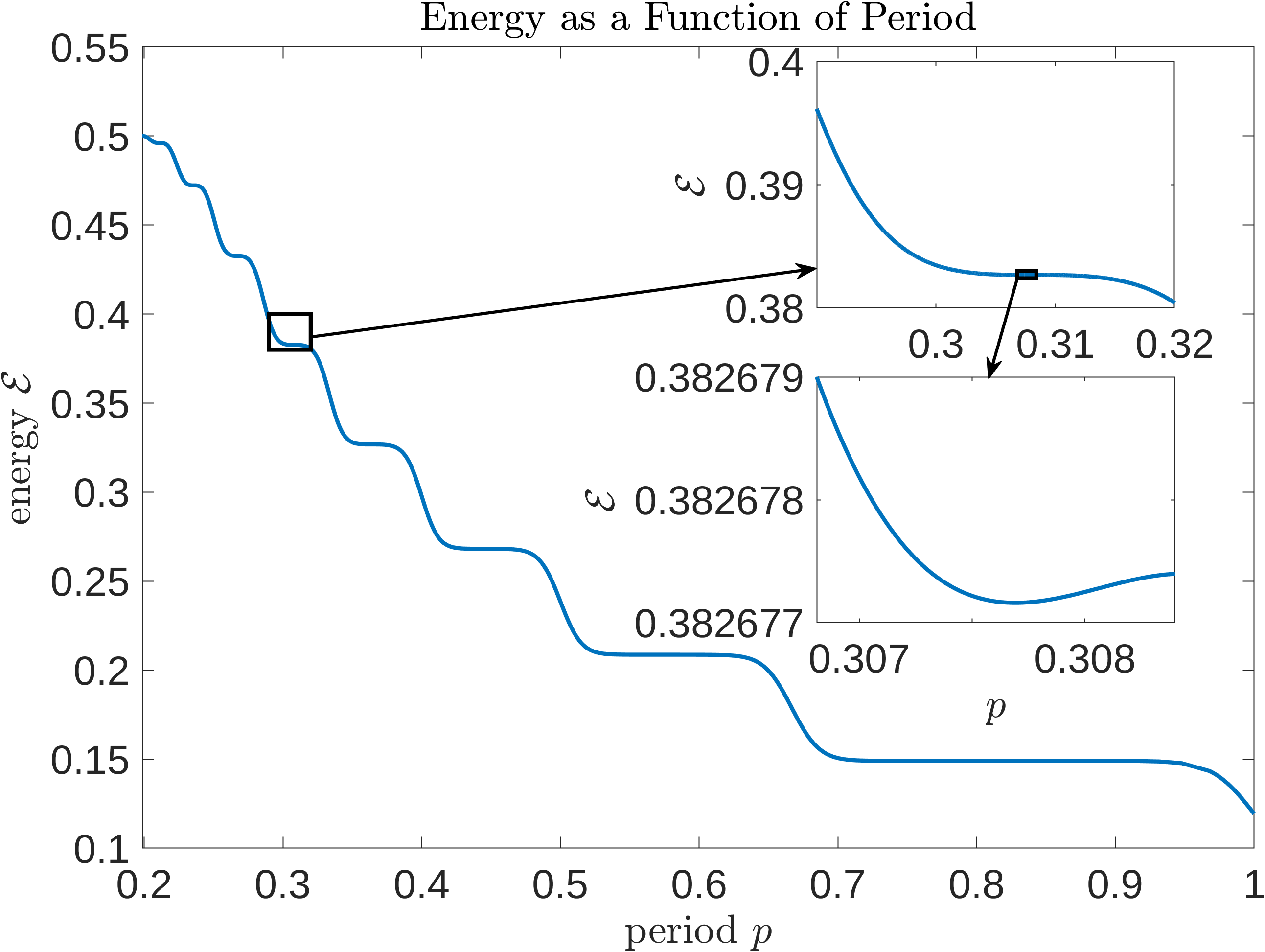}
\caption{\label{EofPzoomed} Plot of $\mathcal{E} (p)$, with zoom-ins at $p \in [0.29, 0.32]$ and $p \in [0.3068, 0.3084]$ inset.  Computed with $\alpha = \beta = 1$, $\kappa = 0.001$, and $L = 1$. We see that, as discussed in Section \ref{MeasureofCoursening}, $\mathcal{E} (p)$ can briefly lose
monotonicity.} 
\end{figure}


We end this section by briefly contrasting our measure of coarseness with 
the measure introduced by Kohn and Otto in \cite{KO02}, stated above as 
Definition \ref{Kohn-Otto-Definition}. In particular, in the one-dimensional 
setting, our measure has the following advantages: (1) we can assign a length scale to any 
function $\phi \in C^1 ([-L, +L])$ as long as $E (\phi) \in  (E_{\min}, E_{\max}]$; 
(2) since our measure of length is directly linked to solutions of 
\eqref{ch}, it serves as a more precise indicator of the extent to which 
the associated solution $\phi (x, t)$ has evolved toward its final 
asymptotic state; and (3) our measure is more readily computed (by \eqref{pseudoinverse_eq}).

\section{Coarsening Rates Models}
\label{coarsening-rates-section}

In Section \ref{sec-computational-results}, we will compare coarsening rates for 
solutions of \eqref{ch} computed 
in three different ways: (1) by direct numerical 
integration of \eqref{ch}; (2) by a long-time coarsening 
model due to Langer \cite{L71}; and (3) by a
coarsening model due to one of the authors \cite{H11}. 
In preparation for that, we now provide background information 
on each of these approaches. 

{\bf Computation.} For our direct computations, we will
initialize the flow with random perturbations of the 
homogeneous configuration $\phi_0 (x) \equiv 0$. We 
then solve forward until the energy $E(\phi (x, t))$
reaches the spinodal energy $E_s$, and it is from this 
point that we compare energies computed in three 
different ways. 

{\bf Langer's relation.} 
In \cite{L71}, Langer employs a statistical development to capture thermal 
fluctuations driving phase separation, and arrives at a straightforward
equation for the evolution of a coarseness measure $\ell$ as a function
of time. In the current framework, we can interpret equation (6.26)
in \cite{L71} as an equation for the period $p$ as a function of 
time, 
\begin{equation} \label{langer_period}
p (t) = p_0 + \sqrt{\frac{2 \kappa}{\beta}} 
\ln \Big(1 + \frac{16 \beta^2 (t-t_0)}{\kappa} e^{-\frac{p_0}{\sqrt{2\kappa/\beta}}} \Big).
\end{equation}
Here, we take the period to be $p_0$ at time $t_0$.

\begin{remark} \label{Langer626remark}
Equation (6.26) in \cite{L71} is 
\begin{equation} \label{Langer626}
\ell (t) = \ell_0 + \frac{\xi}{2} 
\ln \Big(1 + \frac{32 t}{\tau_0} e^{- 2 \ell_0/\xi} \Big),     
\end{equation}
where $\ell (t) = p(t)/2$ (with $p(t)$ as in the current 
analysis; see Figure 4 in \cite{L71}). In obtaining 
\eqref{langer_period}, we have chosen Langer's constants
$\Gamma$, $a$, $d$, $\kappa_b$, and $T$ so that 
\begin{equation*}
    \frac{\Gamma a^{2+d}}{2 \kappa_b T} = 1,
\end{equation*}
in which case equation (3.4) in \cite{L71} becomes
(in the notation of \cite{L71})
\begin{equation*}
    \bar{\eta}_t = (- \epsilon_0 \xi_0^2 \bar{\eta}_{xx} + F' (\bar{\eta}))_{xx}.
\end{equation*}
In particular, we obtain our equation \eqref{ch} with 
$M(\phi) \equiv 1$ and $\kappa = \epsilon_0 \xi_0^2$. The 
combination $\epsilon_0 \xi_0^2$ always appears together in the 
development of \cite{L71}, and is replaced below with the constant  
$\kappa$.
(Here, we have denoted Boltzmann's constant $\kappa_b$ to 
distinguish it from our $\kappa$; Langer denotes it by 
$\kappa$.) In this case, $\tau_0 = 2 \kappa/\beta^2$ and $\xi = \sqrt{2\kappa/\beta}$,
so that \eqref{Langer626} becomes  
\begin{equation*} 
\ell (t) = \ell_0 + \sqrt{\frac{\kappa}{2 \beta}} 
\ln \Big(1 + \frac{16 \beta^2 t}{\kappa} e^{- 2 \ell_0/\sqrt{2\kappa/\beta}} \Big).    
\end{equation*}
Our \eqref{langer_period} is obtained by multiplying this last expression by 
$2$, replacing $2 \ell_0$ with $p_0$, and allowing the initial time 
to be $t_0 > 0$. (Langer takes the initial time to be $0$ by convention, 
initiating late-stage dynamics at $t = 0$.)
\end{remark}

{\bf The approach of \cite{H11}}. A drawback of Langer's approach 
in \cite{L71} is that approximations are made that require the 
coarsening process to be at an asymptotically late stage. In 
particular, Langer approximates late-stage steady-state periodic solutions 
by piecing together enriched regions with transitions taken from 
kink and antikink solutions. In \cite{H11}, this approach is refined
by the use of exact periodic solutions $\bar{\phi} (x; a)$ in place 
of the approximate solutions, with the further advantage that these 
periodic solutions can serve to approximate the dynamics as early as 
the spinodal time. Following this replacement of the approximate
solutions with exact solutions, the method of \cite{H11} follows
Langer's approach of linearizing and determining the coarsening 
rate from the eigenvalues of the resulting eigenvalue problem.
Before describing the model obtained in this way, we briefly summarize 
an efficient method for computing these eigenvalues. First, if 
\eqref{ch} (with $M \equiv 1$) is linearized
about $\bar{\phi} (x; a)$, with $\phi = \bar{\phi} + v$,
we obtain the perturbation equation 
\begin{equation*}
    v_t = (- \kappa v_{xx} + F'' (\bar{\phi})v)_{xx},
\end{equation*}
and the corresponding eigenvalue problem 
\begin{equation} \label{CH-EV-eqn}
    (- \kappa \psi'' + F''(\bar{\phi}) \psi)'' = \lambda \psi.
\end{equation}

As a starting point, we can explicitly compute the right-most
eigenvalue for the limiting case with amplitude $a = 0$, 
corresponding with $\bar{\phi} (x; 0) \equiv 0$. In this
case, \eqref{CH-EV-eqn} becomes 
\begin{equation*}
- \kappa \psi'''' - \beta \psi'' = \lambda \psi,
\end{equation*}
where we have observed that $F'' (0) = - \beta$. 
Any non-imaginary roots of the characteristic polynomial\footnote{Note that, even if one looks for eigenfunctions in the space of tempered distributions, after taking the Fourier transform, one finds that the support of $\widehat{\psi}$ is at most four points, corresponding to solutions of the characteristic polynomial. Hence, the only eigenfunctions are given by the characteristic polynomial method.} for this equation correspond to eigenfunctions which are unbounded as $|x|\rightarrow\infty$.   Moreover, we rule out repeated-root solutions, since these are necessarily unbounded on $\mathbb{R}$.  Hence, we look for eigenfunctions of the form
$\psi (x) = e^{i \xi x}$ for some 
$\xi \in \mathbb{R}$, which yields
\begin{equation*}
    - \kappa \xi^4 + \beta \xi^2 = \lambda.
\end{equation*}
We find that the maximum of $\lambda (\xi)$ 
occurs at $\xi = \pm \sqrt{\beta/(2\kappa)}$, with 
\begin{equation*}
    \lambda_{\max}  = \frac{\beta^2}{4\kappa}.
\end{equation*}
For the specific values $\kappa = 0.001$ and $\beta = 1$,
this is $\lambda_{\max} = 250$.

More generally, in \cite{H11} (which adapts the approach 
of \cite{Gardner1993, Gardner1997} to the current setting), 
Floquet theory is used to characterize the leading eigenvalue associated
with any periodic solution $\bar{\phi} (x; a)$. 
To understand how this works, we first express 
the eigenvalue problem \eqref{CH-EV-eqn}
as a first-order system, setting $y = (y_1, y_2, y_3, y_4)^T$,
with $y_j = \psi^{(j)}$, $j = 0, 1, 2, 3$. Then $y$ solves 
the system 
\begin{equation} \label{ev-ode}
    y' = \mathbb{A} (x; \lambda) y,
    \quad 
    \mathbb{A} (x; \lambda)
    = \begin{pmatrix}
        0 & 1 & 0 & 0 \\
        0 & 0 & 1 & 0 \\
        0 & 0 & 0 & 1 \\
        (b''(x) - \lambda)/\kappa & 2b'(x)/\kappa & b(x)/\kappa & 0
    \end{pmatrix},
\end{equation}
where $b(x) = F''(\bar{\phi} (x; a))$.
We can compute a fundamental matrix for \eqref{ev-ode},
\begin{equation*}
    \Phi' = \mathbb{A} (x; \lambda) \Phi,
    \quad \Phi (0; \lambda) = I_4,
\end{equation*}
and the monodromy matrix is defined to be $M(\lambda; p) := \Phi (p; \lambda)$.
The Evans function is 
\begin{equation}
    D(\lambda, \xi) 
    = \det (M(\lambda; p) - e^{i \xi p} I_4), 
\end{equation}
and in this setting $\lambda$ is an eigenvalue of \eqref{CH-EV-eqn}
if and only if $D (\lambda, \xi) = 0$ for some $\xi \in \mathbb{R}$.
(See \cite{Gardner1993, Gardner1997} for details on the 
Evans function for periodic solutions generally, and \cite{H09, H11} for 
specialization to the current setting.)

In practice, we can use this to compute leading eigenvalues in 
the following relatively efficient way. First, we know that 
for $a = 0$, with corresponding minimum period $p_{\min}$ 
(from \eqref{minimim-period}), the leading eigenvalue is 
$\lambda_{\max}  = \frac{\beta^2}{4\kappa}$. We now increment
the value $a$ by some small step $\Delta a$, and compute the 
corresponding period $p(a+\Delta a)$, using \eqref{a-to-p}.
The leading eigenvalue for $\bar{\phi} (x; a+ \Delta a)$ will
be just below the leading eigenvalue for $\bar{\phi} (x; a)$,
so we search for zeros of $D (\lambda, \xi)$ with values of 
$\lambda$ just below $\lambda_{\max}$. Repeating this process
iteratively we obtain the leading eigenvalues. 
For $\alpha = 1$, $\beta = 1$, and $\kappa = 0.001$, a plot of 
leading eigenvalues as a function of amplitude $a$ is depicted 
in Figure \ref{evplot1} for two different values of $\kappa$. 

\begin{figure}[ht] 
\begin{center}
\includegraphics[width=12cm,height=8.2cm]{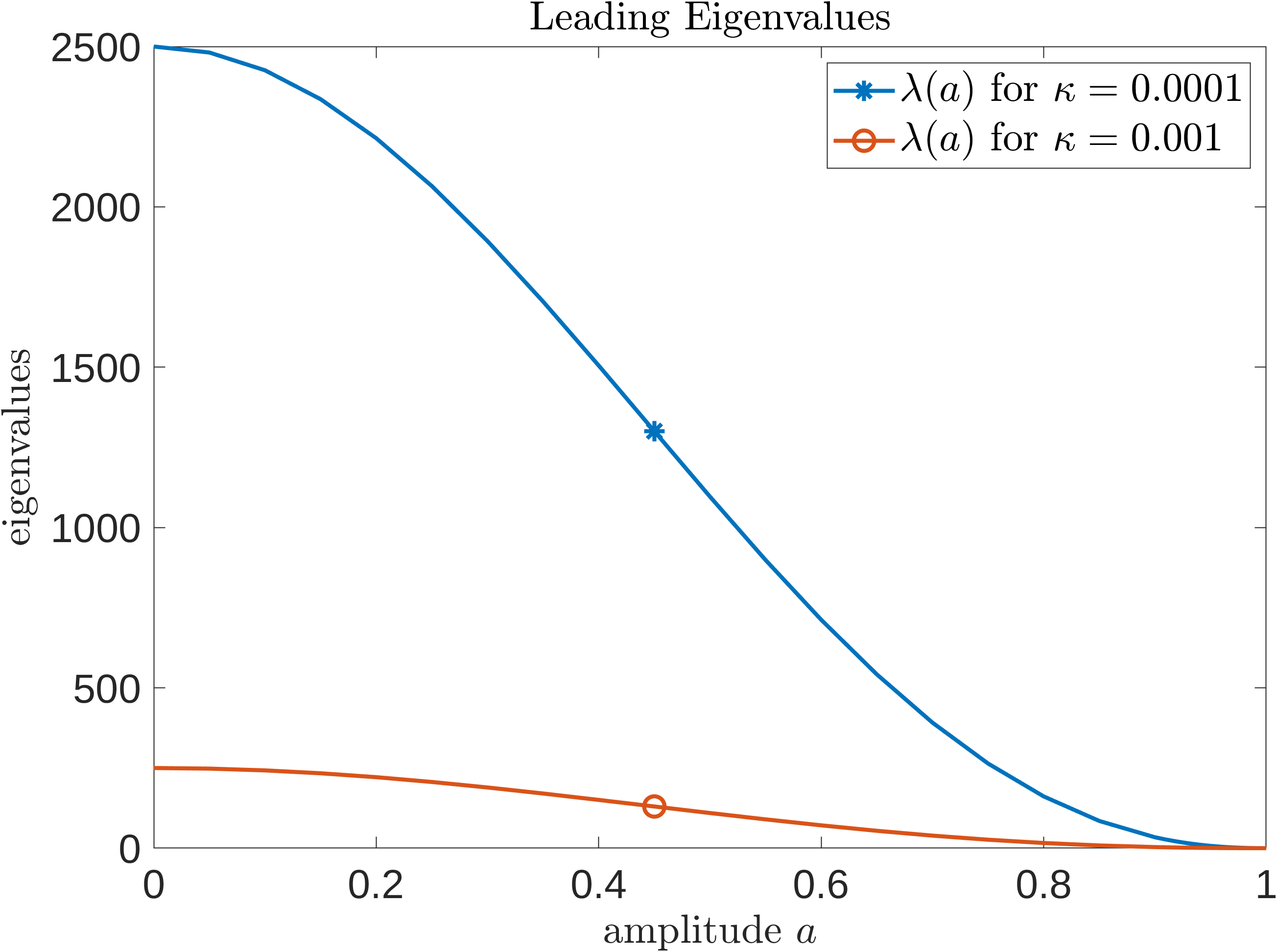}
\end{center}
\caption{Plot of leading eigenvalues versus amplitude, 
computed with $\alpha = \beta = 1$, and $\kappa = 0.001$ and $\kappa = 0.0001$. 
\label{evplot1}}
\end{figure}

According to \cite{H11}, the period $p(t)$ associated with 
an evolving solution of \eqref{ch} evolves approximately 
according to the relation 
\begin{equation} \label{eigenvalue-ode}
    \frac{dp}{dt}
    = \lambda_{\max} (p) p,
    \quad p(t_0) = p_0,
\end{equation}
where $\lambda_{\max} (p)$ denotes the maximum eigenvalue 
associated with the specified periodic solution $\bar{\phi} (x;a(p))$.

\begin{remark}\label{half_ev_remark}
In the derivation of \eqref{eigenvalue-ode} in \cite{H11}, 
a small modification to Langer's framework is employed. Precisely,
the equation arising from a more faithful adaptation of 
Langer's original relations is
\begin{equation} \label{eigenvalue-ode-langer}
    \frac{dp}{dt}
    = \frac12\lambda_{\max} (p) p,
    \quad p(t_0) = p_0.
\end{equation}
As a side note to our analysis, we will compare how these
two possible methods compare with numerically generated 
solutions. 

If Langer's method and the eigenvalue method are both initialized 
by the minimum period $p_{\min}$ (from \eqref{minimim-period}), 
the evolution in time proceeds as in Figure \ref{max-coarsening-figure}
(computed for $\alpha = \beta = 1$, $\kappa = 0.001$). 
Using $\mathcal{E} (p)$ to map periods to energies, we can evolve
energy as a function of time for both Langer's method and the 
eigenvalue method. This is depicted in Figures \ref{max-energy-figure} and \ref{max-coarsening-figure}. 
In order to see that the primary difference between the two approaches
is a matter of scaling, we included similar plots using  \eqref{eigenvalue-ode-langer} in place of \eqref{eigenvalue-ode}, labeled with the tag, ``(1/2 factor)'' in Figures \ref{max-energy-figure} and \ref{max-coarsening-figure}.
\end{remark}





\begin{figure}[ht] 
\begin{minipage}[T]{0.48\textwidth}
\centering
\begin{subfigure}[T]{1\textwidth}
\includegraphics[width=1\textwidth]{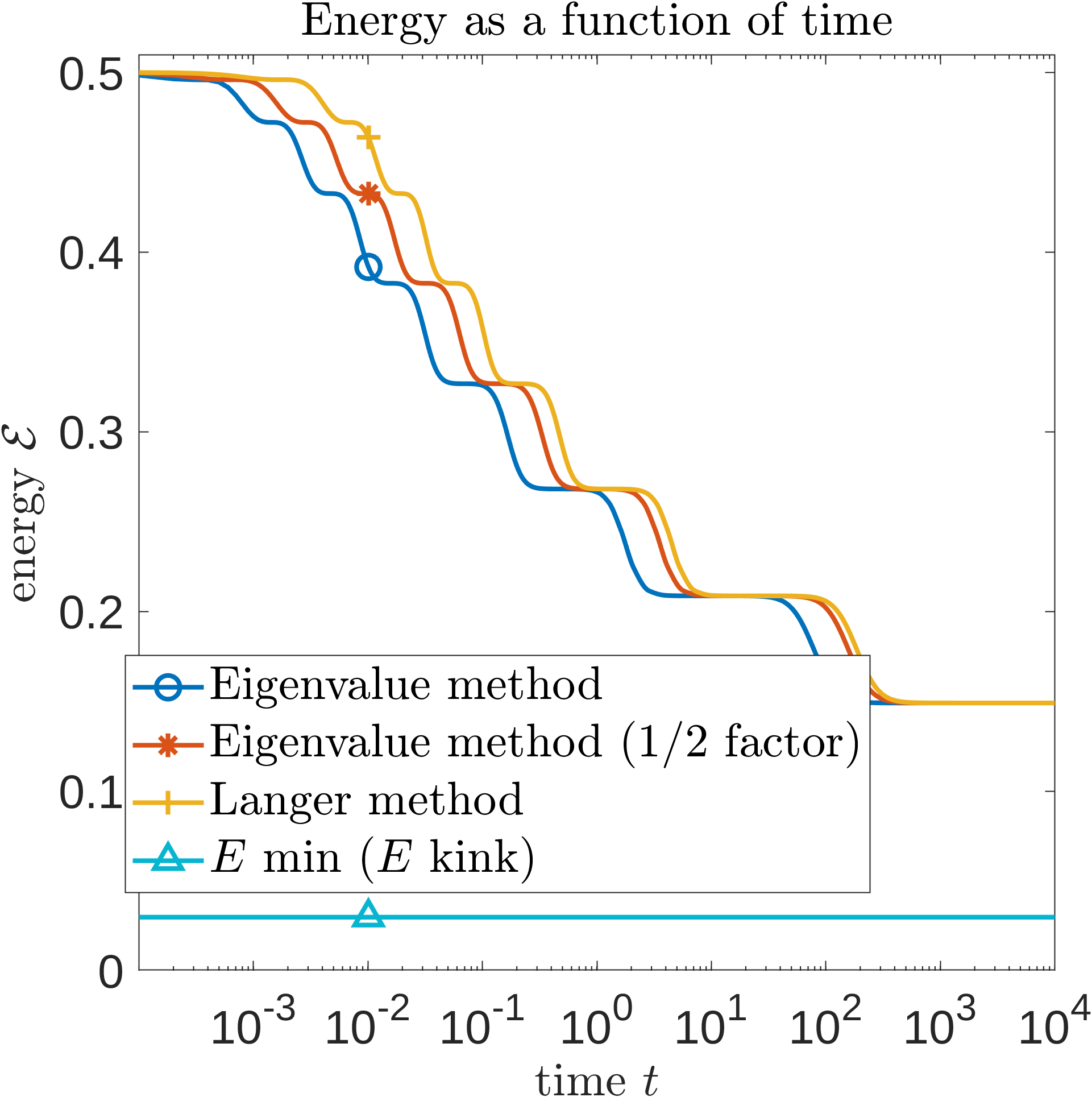}
\caption{\label{max-energy-figure}}
\end{subfigure}
\end{minipage}
\hfill 
\begin{minipage}[T]{0.48\textwidth}
\centering
\begin{subfigure}[T]{\textwidth}
\includegraphics[width=1\textwidth]{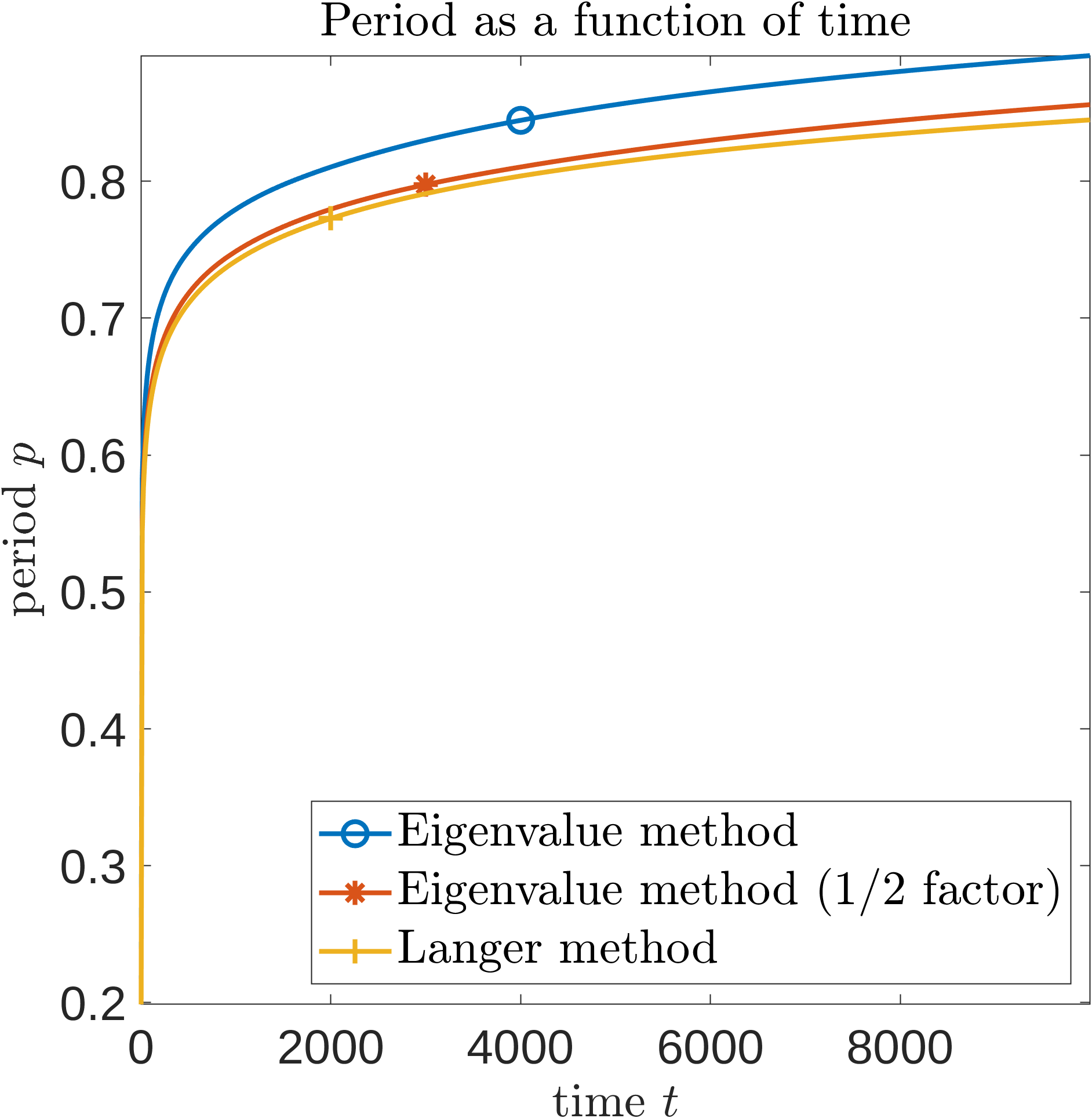}
\caption{\label{max-coarsening-figure}}
\end{subfigure}
\end{minipage}
\caption{\label{2x2}Evolution of energies (a) and periods (b) via Langer's method and the eigenvalue method.  For the plots labeled ``Eigenvalue method,'' equation \eqref{eigenvalue-ode} was used.  For the plots with the tag ``(1/2 factor)'', equation \eqref{eigenvalue-ode-langer} (with the extra $\frac12$-factor) was used (\textit{cf.} Remark \ref{half_ev_remark}).}
\end{figure}

\FloatBarrier

\subsection{Varying \texorpdfstring{$\kappa$}{}}
In this section, we observe that it is straightforward to vary 
$\kappa$ in the above energy calculations. First, we note that 
for a fixed interval $[-L, +L]$ the maximum 
energy specified in \eqref{maximum-energy} does not depend on $\kappa$. 
On the other hand, the minimum energy specified in 
Item (iv) of Proposition \ref{specific-F-proposition} 
is proportional to $\sqrt{\kappa}$. 
For evolution on bounded domains, transition layers are 
typically removed through the boundary in pairs, so we 
expect to see energy drops in steps of size 
$2 E_{\min}$. For $\alpha = 1$, $\beta = 1$, and $L = 1$, 
if $\kappa = 0.001$ then $2 E_{\min} = 0.0596$, 
and likewise if $\kappa = 0.0001$, then $2 E_{\min} = 0.0189$, and if 
$\kappa = 0.00001$, then $2 E_{\min} = 0.0060$. Since the energy 
declines in steps of these sizes, energy plots such as the two 
depicted in Figure \ref{energies_kappa3} have less pronounced steps
for smaller values of $\kappa$. 

For Langer's approach, dependence on $\kappa$ is explicit in 
\eqref{langer_period}, and so changes in $\kappa$ are readily 
accommodated. For the eigenvalue approach, we need to identify 
how the leading eigenvalues vary with $\kappa$. To this end, we 
fix a choice of $F$
with the form \eqref{quarticF}, and we observe 
that if $\bar{\phi} (x)$
denotes a periodic solution of \eqref{ch} obtained with $\kappa = 1$, then 
for any $\kappa > 0$, $\bar{\phi}^{\kappa} (x) := \bar{\phi} (x/\sqrt{\kappa})$ is a 
periodic solution of \eqref{ch} obtained with the value 
$\kappa$. Upon linearization of \eqref{ch} about $\bar{\phi}^{\kappa} (x)$,
we arrive at the eigenvalue problem 
\begin{equation*}
    (-\kappa \psi'' + F'' (\bar{\phi}^\kappa (x)) \psi)'' = \lambda \psi.
\end{equation*}
We can express this equation as 
\begin{equation*}
    (-\kappa \psi'' + b(x/\sqrt{\kappa}) \psi)'' = \lambda \psi,
\end{equation*}
where $b (x/\sqrt{\kappa}) = F'' (\bar{\phi} (x/\sqrt{\kappa}))$, 
and we can also express this as 
\begin{equation*}
    - \kappa \psi'''' + \frac{1}{\kappa} b''(x/\sqrt{\kappa}) \psi 
    + \frac{2}{\sqrt{\kappa}} b' (x/\sqrt{\kappa}) \psi' 
    + b(x/\sqrt{\kappa}) \psi'' = \lambda \psi. 
\end{equation*}
At this point, we make the change of variables
\begin{equation*}
    y = \frac{x}{\sqrt{\kappa}}, \quad
    \Psi (y) = \psi(x) \implies \psi^{(k)} (x) = \frac{1}{\kappa^{k/2}} \Psi^{(k)} (y), 
    \quad k = 1,2,3,\dots.
\end{equation*}
Upon substitution, we see that the equation for $\Psi (y)$ is 
\begin{equation*}
    -(\Psi'' + b(y) \Psi)'' = \kappa \lambda \Psi.
\end{equation*}
We take from this that $\lambda$ is an eigenvalue for 
the equation with general $\kappa$ if and only if $\kappa \lambda$
is an eigenvalue for the equation with $\kappa = 1$. This observation 
allows us to compute leading eigenvalues for any fixed $\kappa > 0$,
and obtain leading eigenvalues associated with all other values 
of $\kappa$ by an appropriate scaling argument. To make this 
precise, suppose that for a specific value $\kappa_0$, we compute
the leading eigenvalue $\lambda_0$ associated with the periodic 
solution with amplitude $a$. Then $\kappa_0 \lambda_0$ will be 
the leading eigenvalue associated with the periodic solution 
with amplitude $a$ arising as a solution to \eqref{ch} with 
$\kappa = 1$. Correspondingly, the leading eigenvalue associated
with the periodic solution with amplitude $a$ arising as a solution
to \eqref{ch} with any general value $\kappa > 0$ will be 
$(\kappa_0/\kappa) \lambda_0$. In fact, we have already seen an 
example of this scaling effect in Figure \ref{evplot1}. 
 
In Figure \ref{energies_kappa3}, we apply these ideas 
to generate for comparison plots of energy as a function 
of time for both Langer's method and the eigenvalue method
for three values of $\kappa$, namely $\kappa = 0.001$, 
$\kappa = 0.0001$, and $\kappa = 0.00001$. In each case, the energy at 
$t = 0$ is the same maximum energy $E_{\max} = 0.5$, 
but the plots are depicted starting with $t = 0.01$,
at which point the energies have already declined at 
varying rates. (These calculations use \eqref{eigenvalue-ode}
rather than \eqref{eigenvalue-ode-langer}.)

\begin{figure}[ht] 
\begin{center}
\includegraphics[width=12cm,height=8.2cm]{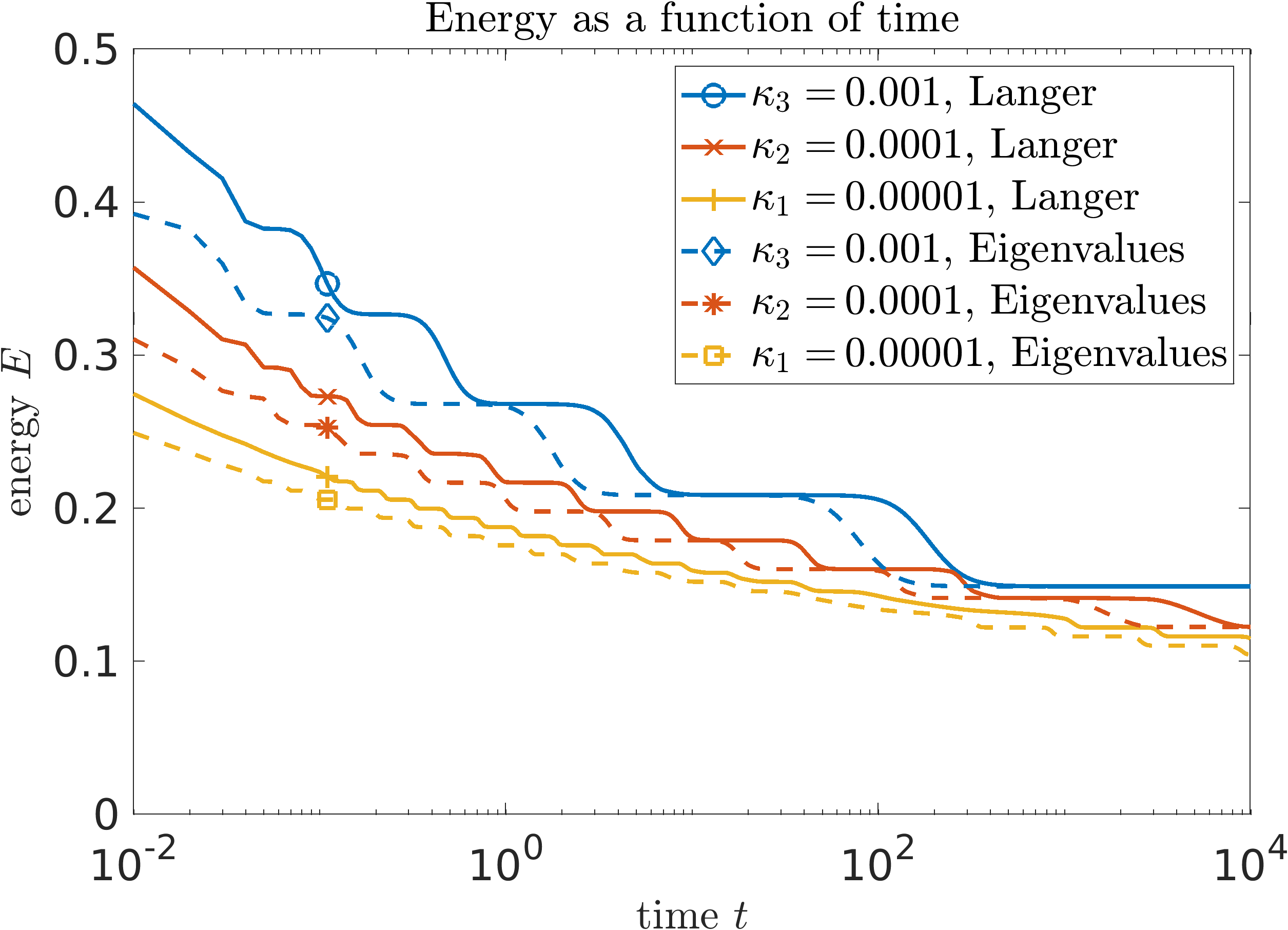}
\end{center}
\caption{Evolution of energies via Langer's method and the eigenvalue method for various values $\kappa$. 
\label{energies_kappa3}}
\end{figure}

\section{Computational Results}
\label{sec-computational-results}

In this section, we consider constant mobility $M\equiv 1$ since the qualitative dynamics do not vary much as long as the mobility is
non-degenerate.

\subsection{Overview of the numerical methods}
\label{overview-numerical-section}

Here, we give a brief summary of the numerical methods used throughout this analysis.  
As much as possible, we have used standard well-established methods, as our goal is not to 
focus on developing numerical methods, but on understanding dynamical phenomena.

\begin{remark}[Periodic Domain vs.\ Full Space]
While the primary focus of the present work is in the full space $\mathbb{R}$, for the numerical simulations, we work in the periodic domain $[-1,1]$. 
Working in a periodic domain yields spectral accuracy for smooth, localized profiles without introducing boundary layers or buffer‐zones.  By contrast, so-called ``full‐space discretization'' methods  (e.g., sinc‐type methods on $\mathbb{R}$, or a nonperiodic truncation to $[-L_*,L_*]$ as $L_*\rightarrow\infty$, or a remapping of $\mathbb{R}$ to a finite interval) come with their own problems, which we now discuss briefly. 
Sinc methods (see, e.g., \cite{Lund_Bowers_1992_sinc_methods,Stenger_2016_sinc_handbook}) can enjoy spectral accuracy, but to achieve this, one must delicately choose both the grid spacing and the number of nodes carefully to match the unknown strip of analyticity, which poses non-trivial difficulties that would take away from the focus of the present work.  Moreover, the differentiation matrix is dense, and can impose severe timestep restrictions.  Finally, sinc methods require the solution to decay to zero as $|x|\rightarrow\infty$, and hence to recover, e.g., the kink solution \eqref{kink-solution}, one must impose some type of smooth cut-off function anyway.  Similar problems arise with remapping $\mathbb{R}$ to a finite interval, which moreover introduces non-constant coefficients and can introduce undesirable ``endpoint clustering'' (if one discretizes with, e.g., Chebyshev polynomials).  Truncation methods must impose artificial boundary conditions at $x=\pm L_*$ leading to loss of accuracy, and can introduce extra tuning parameters (such as the truncation length $L_*$ or the placement of ghost nodes) and often exhibit slower convergence and stability issues due to the nonlinear terms.  Hence, for the sake of both simplicity and accuracy, our simulations are carried out in the periodic domain, with the usual adjustments in understanding that come with considering dynamics in full space vs. periodic boundary conditions.
Note also that we still consider comparison with the full-space ``kink'' solution \eqref{kink-solution}. In the parameter regimes we consider, this solution has interfaces whose deviations from $\pm1$ decay exponentially fast as $|x|\to\infty$.  When we restrict to $[-1,1]$ with periodic boundary conditions, 
truncating to $[-1,1]$ with periodic boundary conditions incurs an $L^\infty$ error of order
$\mathcal{O}(\exp(-2\sqrt{\tfrac{\beta}{2\kappa}})),$
which is numerically negligible $(\approx 10^{-20})$ for the $\beta,\kappa$ values used in our simulations.  When the solution evolves, the discontinuity at the periodic boundary quickly vanishes and the solution becomes periodic, but otherwise stays near the kink solution in $L^2$ norm, although the free energy differs by roughly $E_{\min}$ due to the additional transition layer at the periodic boundary. Thus, although our analysis is focused on full space, periodic computations on $[-1,1]$ can give a reasonable picture of the dynamics, and using this setting has at least some advantages over ``full-space discretizations'' as described above. 
\end{remark}

All simulations were run using MATLAB version 2024a.  The spatial discretization was done using standard spectral methods, based on MATLAB's \texttt{fft} and \texttt{ifft} (the latter computed using the \texttt{{\textquotesingle}symmetric\textquotesingle} option).  Time-stepping was handled by a semi-implicit Euler method.  
Eyre's convex splitting method \cite{Eyre_1998,Eyre_1997} was employed (in particular, algorithm 5, the ``Linearly Stabilized Splitting Scheme'' proposed in \cite{Eyre_1997}), with the cubic term computed using co-location (i.e., multiplication in physical space).  Due to the presence of the cubic term, the highest half of the wave modes were dealiased (i.e., set to zero before being transformed back to physical space).  
Spatial resolution on the domain $[-1,1)$ was chosen to be $N=8192$, giving a spatial stride of $\Delta x=2\pi/N\approx7.6699\times10^{-4}$.  The time-step $\Delta t\approx 9.7656\times 10^{-5}$ was chosen to respect 
standard CFL conditions.  For simplicity, we chose $\alpha=\beta=K=1$. Our smallest choice
of interfacial energy $\kappa\geq0.00001$, was found experimentally using the resolution criterion that the energy spectra of dealiased modes of 
$\phi$ 
must be below machine precision ($\approx2.2204\times10^{-16}$) at all time steps after roughly time $t\approx 0.001$ (obviously, since we are starting with normally-distributed random values at each point, the initial data, and consequentially the first few time steps, are not expected to satisfy this criterion).

We often initialize the equation with ``random'' initial data, which has been evolved until the free energy is just below $0.99 E_{\max}$, our
tolerance taken somewhat arbitrarily to be $10^{-4}$, where $E_{\max}:=\frac{\beta^2 L}{4\alpha}$.  This is carried out in the following way.  First, for the purposes of easy replication of our results, we seed the random number generator (RNG) with seed \verb|0|, then we take normally-distributed mean-zero  data at each of the $N$ points in space with standard deviation \verb|0.1| (i.e., $\mathcal{N}(0,0.1)$).  The result is then transformed with the discrete Fourier transform, where the upper $N/2$ wave modes are removed (for dealiasing the cubic term).  For clarity, the MATLAB code used for this operation is as follows:
\begin{lstlisting}
    rng(0,'twister');
    phi_hat = fft(0.1*randn(1,N));
    phi_hat((N/4+1):(3N/4+1)) = 0;
    phi = ifft(phi_hat,'symmetric');
\end{lstlisting}
In our simulations, we used $N=2^{11}=2048$ and $L=\alpha=\beta=1$.
The RNG was only seeded by zero on the first trial.  On subsequent trials, the seed was determined by whatever the previous state of the system was.  After the initialization of $\phi$ described above, the simulation was then run until free energy satisfied the criterion stated above.  For this last stage, we used a variable time-step: if the free energy was not within the desired tolerance of $0.99 E_{\max}$, the time step was thrown out and recomputed with half the previous time-step.  Once this process was completed, the result was used to initialize our Cahn--Hilliard simulations.  

\subsection{Simulations for the Cahn--Hilliard Equation}
\label{uncoupled-system-section}
We now carry out our direct computations for the Cahn--Hilliard equation, emphasizing the energy 
evolution and associated coarsening dynamics.
In Figure \ref{energies_av}, we depict energies (Figure \ref{energies_av_a}) and periods
(Figure \ref{energies_av_b}) for 50 trials of the Cahn--Hilliard equation---trials in gray, 
average in thicker black---along with energies and periods computed with the two analytical methods
described in Section \ref{coarsening-rates-section}. For the computationally generated values,  
initial data for each trial was computed pseudo-randomly as described  
in Section \ref{overview-numerical-section}, except with the random number generator seeded randomly 
each time, using \texttt{rng(randi(10000))}.  The free energy was computed for each trial at each 
time-step, and the mean of these curves was computed and displayed in Figure \ref{energies_av_a}.  
This is the first direct quantitative comparison that we are aware of for the energies obtained from 
these three approaches, and we interpret the consistency of the results at once as a justification of 
Langer's approach in \cite{L71} (and subsequently as modified in \cite{H11}) and a verification that 
our computations are faithfully capturing dynamics of the energy. We observe that 
the horizontal shelves for the three plots occur approximately at integer multiples of the minimum 
energy, and correspond with the number of transitions in the solution at that point. For the two 
analytic methods, the underlying periodic solutions are centered with a transition passing through 
$x = 0$, so by symmetry each shelf will correspond with an odd multiple of the minimum energy (i.e., 
there will be a transition at $x = 0$ and the same number of transitions to the right and left of 
this). For the computed energies, periodicity requires an even number of transitions, so the horizontal 
shelves occur at even multiples of the minimum energy. Notably, this means that the analytic curves
will approach $E_{\min}$ asymptotically while the computed curves will approach $2 E_{\min}$ 
asymptotically. In particular, it's clear that the analytic curves are only transiently above
the computed curve. 

In Figure \ref{energies_av_b}, we depict the evolution of the periods associated with the energies 
in Figure \ref{energies_av_a}. For the two analytic methods, these periods are computed directly 
from \eqref{langer_period} (Langer's method) and \eqref{eigenvalue-ode-langer} (eigenvalue method), 
while for the computational approach the periods are computed from the energies via the pseudoinverse 
as specified in Definition \ref{pseudoinverse} (though see the caption of Figure \ref{energies_av} 
for a note on this). The relatively large discrepancy in the latter stages of the evolution is the 
result of long shelves in $E (p)$, as depicted in Figure \ref{energy-period-figure}. As noted in the 
previous paragraph, the analytic energy will asymptotically pass below the computed energy, and 
correspondingly the analytic values for coarseness will become larger than the computed values.  

Since the two analytic models were developed by considering \eqref{ch} on $\mathbb{R}$,
it's natural to expect a better correspondence for larger intervals $[-L, L]$. In order to see that
this is indeed the case, we include in Figure \ref{energies_av_big_interval_short_time} plots of
energies and periods computed for $L = 10$ (and keeping $\kappa = .001$). The horizontal shelves are
again separated by steps of size $2 E_{\min}$, but these steps are now relatively quite small, and we
see close correlation especially between the method of \cite{H11} and the computed values. Once again, 
the coarseness obtained by computation trends slightly below the coarseness obtained from the analytic
models for later times in the figure, but we again note that ultimately this will switch so that
asymptotically the 
coarseness obtained from computation will be slightly above the coarseness obtained from the analytic 
models. 
Unfortunately, the time expected to see this final switch with parameter values $L = 10$ and $\kappa = .001$
is about $10^{18}$, so no such plot was attempted. Rather, we include in Figure \ref{energies_av_small_interval_large_time}
a simple plot obtained for $L = 1$ and the larger value $\kappa = .01$ to illustrate the expected 
time-asymptotic behavior.


\begin{figure}[ht] 
\begin{subfigure}[t]{0.49\textwidth}
\includegraphics[width=1\textwidth]{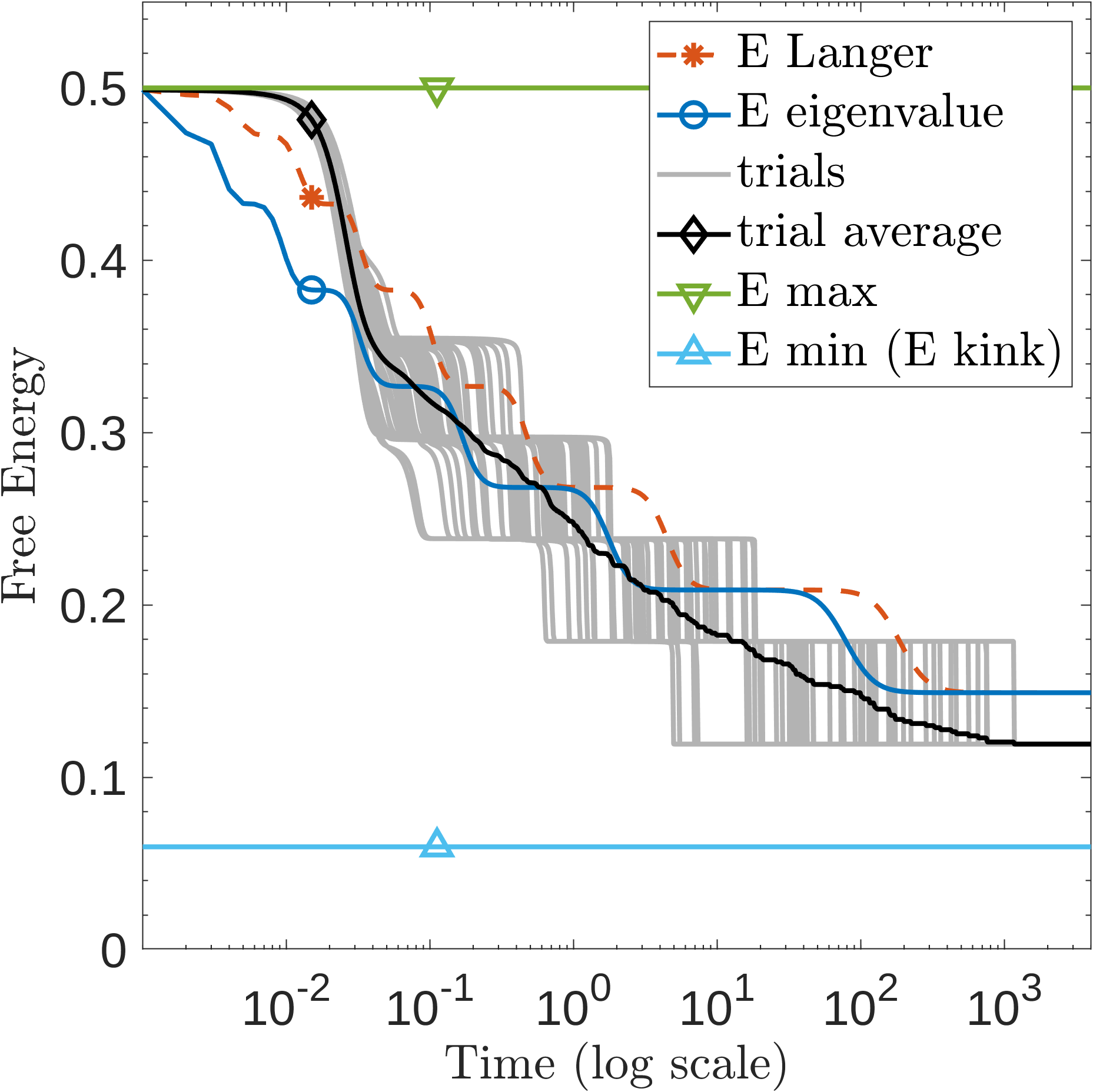}
\caption{\label{energies_av_a} (linear-log scale) Free Energy of $\phi$ vs. time}
%
\end{subfigure}
\begin{subfigure}[t]{0.48\textwidth}
\includegraphics[width=1\textwidth]{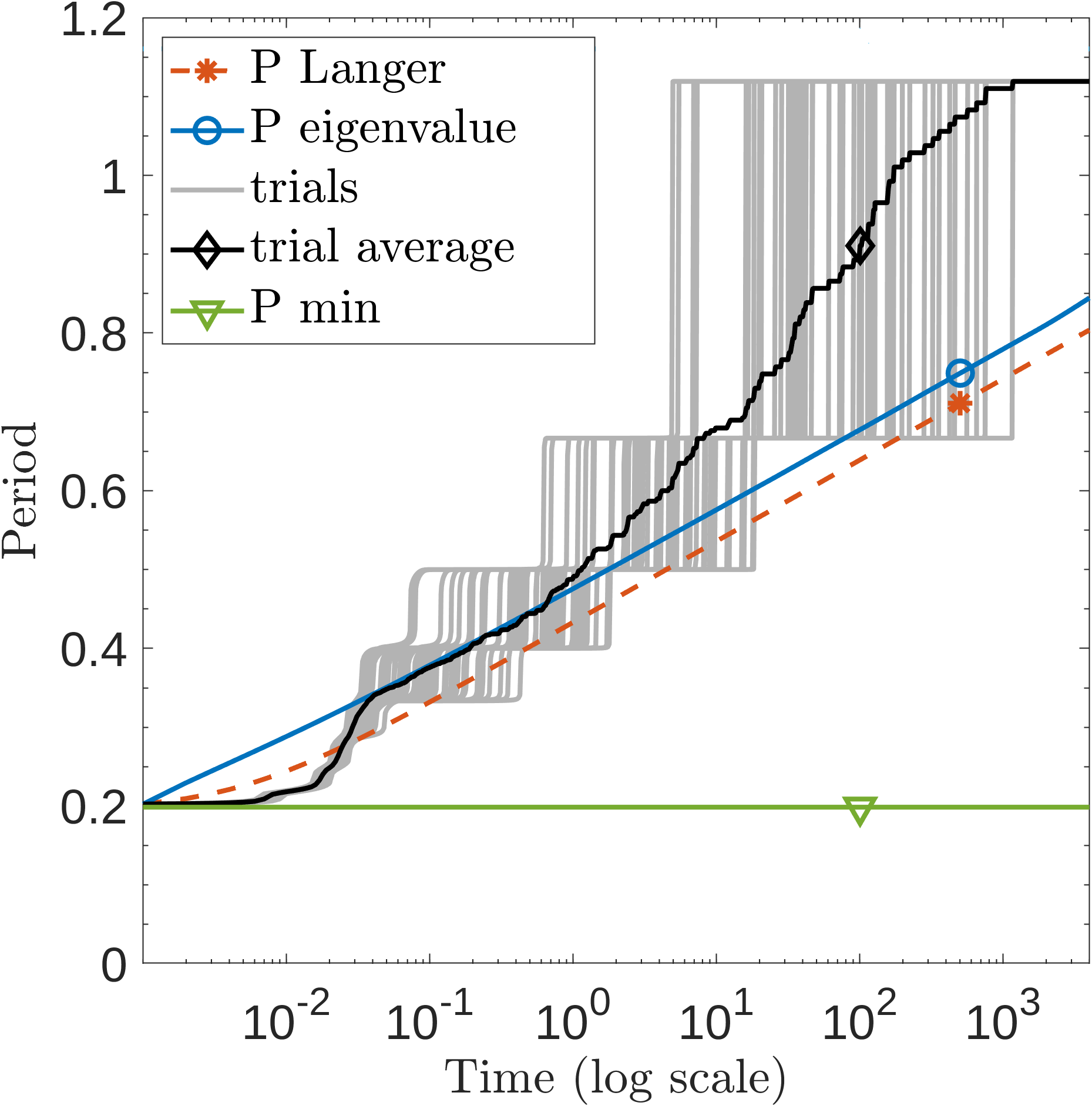}
\caption{\label{energies_av_b}  (linear-log scale) Period of $\phi$ vs. time.}
\end{subfigure}
\caption{\label{energies_av}
Free energy and periods of 50 trials of the Cahn-Hilliard equation ($\kappa=0.001$, $\alpha=\beta=1$, $N=2^{11}=2048$, Domain: $[-1,1]$, $\Delta t=0.001$, final time $T=4000$), and the average of these, along with the maximum free energy \eqref{maximum-energy}, the kink energy \eqref{E_kink}, and the Langer and eigenvalue predictions.
In principle, the computational plots in Figure (b) should be generated by applying the pseudoinverse map 
from Definition \ref{pseudoinverse} to the plots in Figure (a), but in practice we found it 
more efficient to generate a map by interpolating periods as a function of energies. It is clear 
from our discussion in Section \ref{MeasureofCoursening} that the two approaches must give
nearly identical results.} 
\end{figure}

\begin{figure}[ht] 
\begin{subfigure}[t]{0.49\textwidth}
\includegraphics[width=1\textwidth]{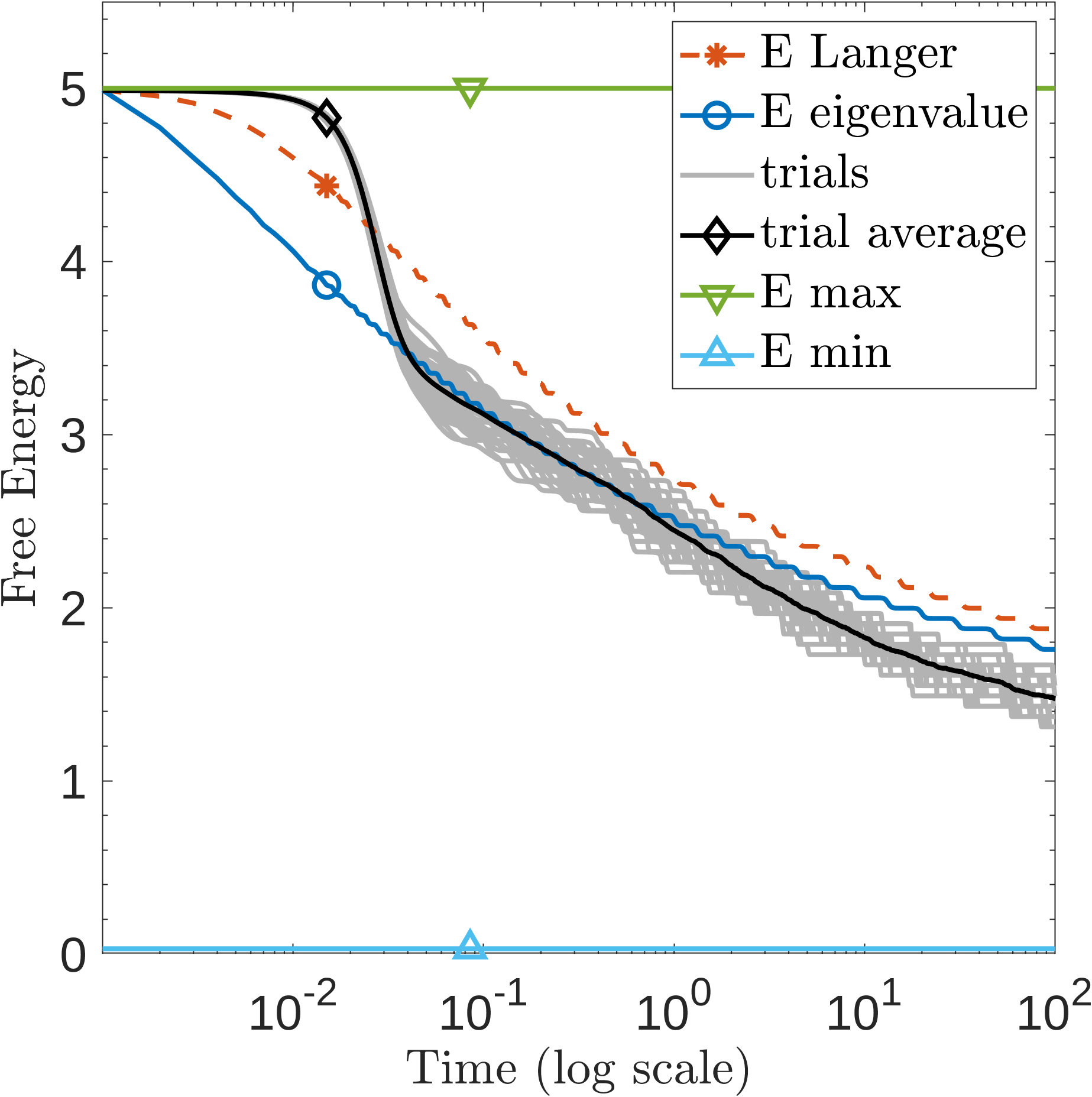}
\caption{\label{energies_av_a_big_interval_short_time} (linear-log scale) Free Energy of $\phi$ vs. time}
%
\end{subfigure}
\begin{subfigure}[t]{0.48\textwidth}
\includegraphics[width=1\textwidth]{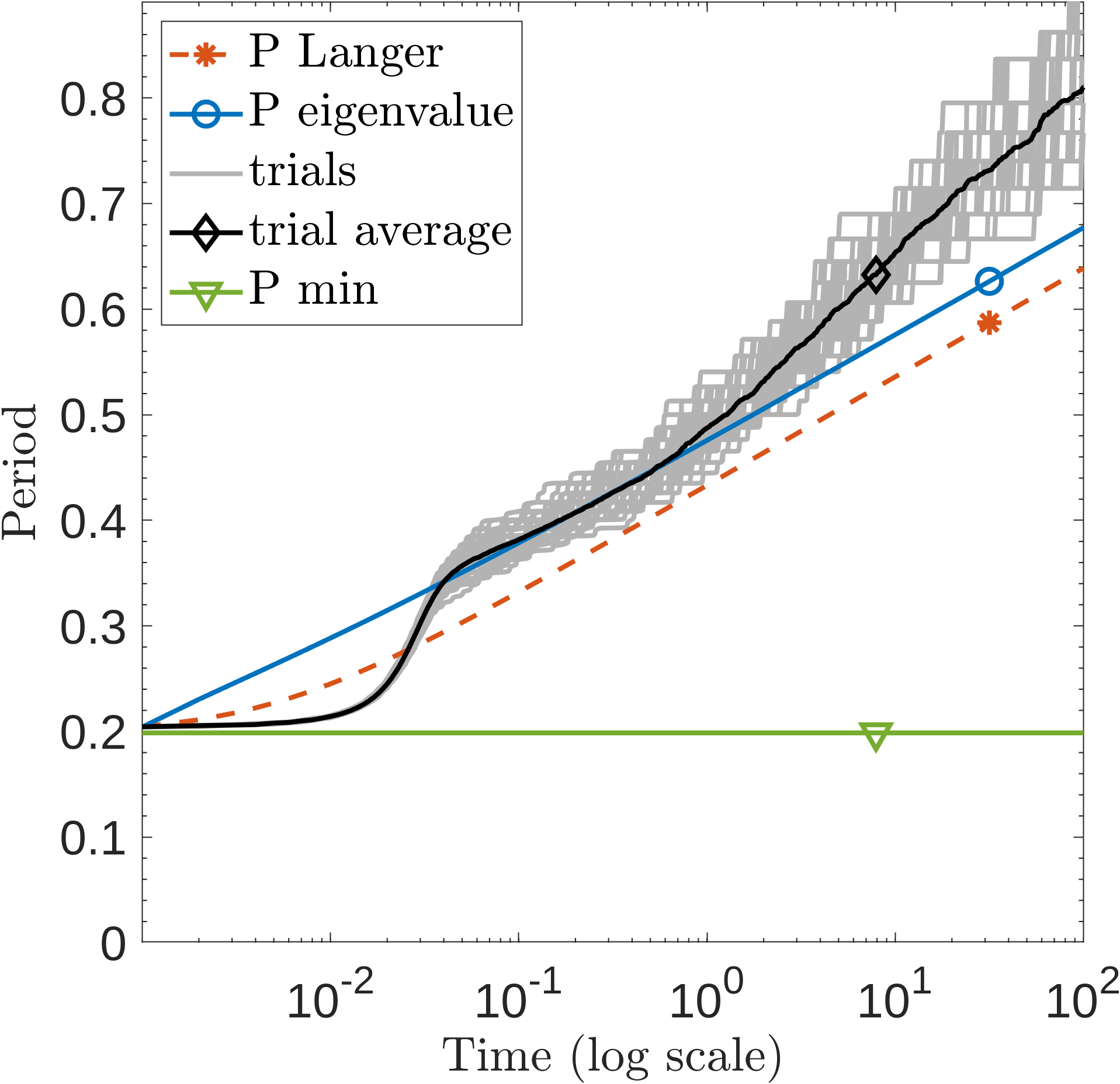}
\caption{\label{energies_av_b_big_interval_short_time}  (linear-log scale) Period of $\phi$ vs. time.}
\end{subfigure}
\caption{\label{energies_av_big_interval_short_time}
Free energy and periods of 50 trials of the Cahn-Hilliard equation ($\kappa=0.001$, $\alpha=\beta=1$, $N=2^{13}=8192$, Domain: $[-10,10]$, $\Delta t=0.001$, final time $T=100$), and the average of these, along with the maximum free energy \eqref{maximum-energy}, the kink energy \eqref{E_kink}, and the Langer and eigenvalue predictions.  
}
\end{figure}

\begin{figure}[ht] 
\begin{subfigure}[t]{0.49\textwidth}
\includegraphics[width=1\textwidth]{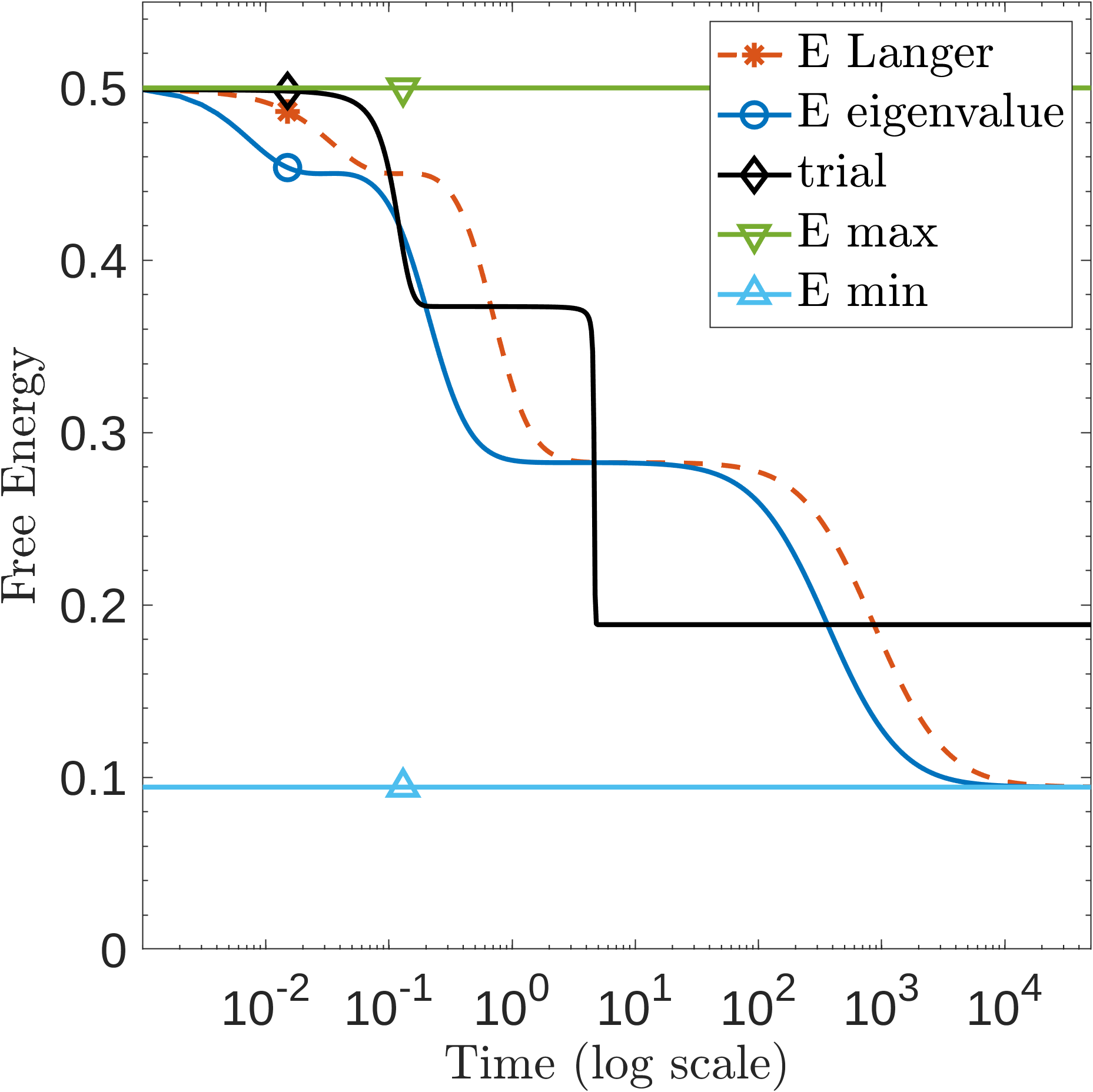}
\caption{\label{energies_av_a_small_interval_large_time} (linear-log scale) Free Energy of $\phi$ vs. time}
%
\end{subfigure}
\begin{subfigure}[t]{0.48\textwidth}
\includegraphics[width=1\textwidth]{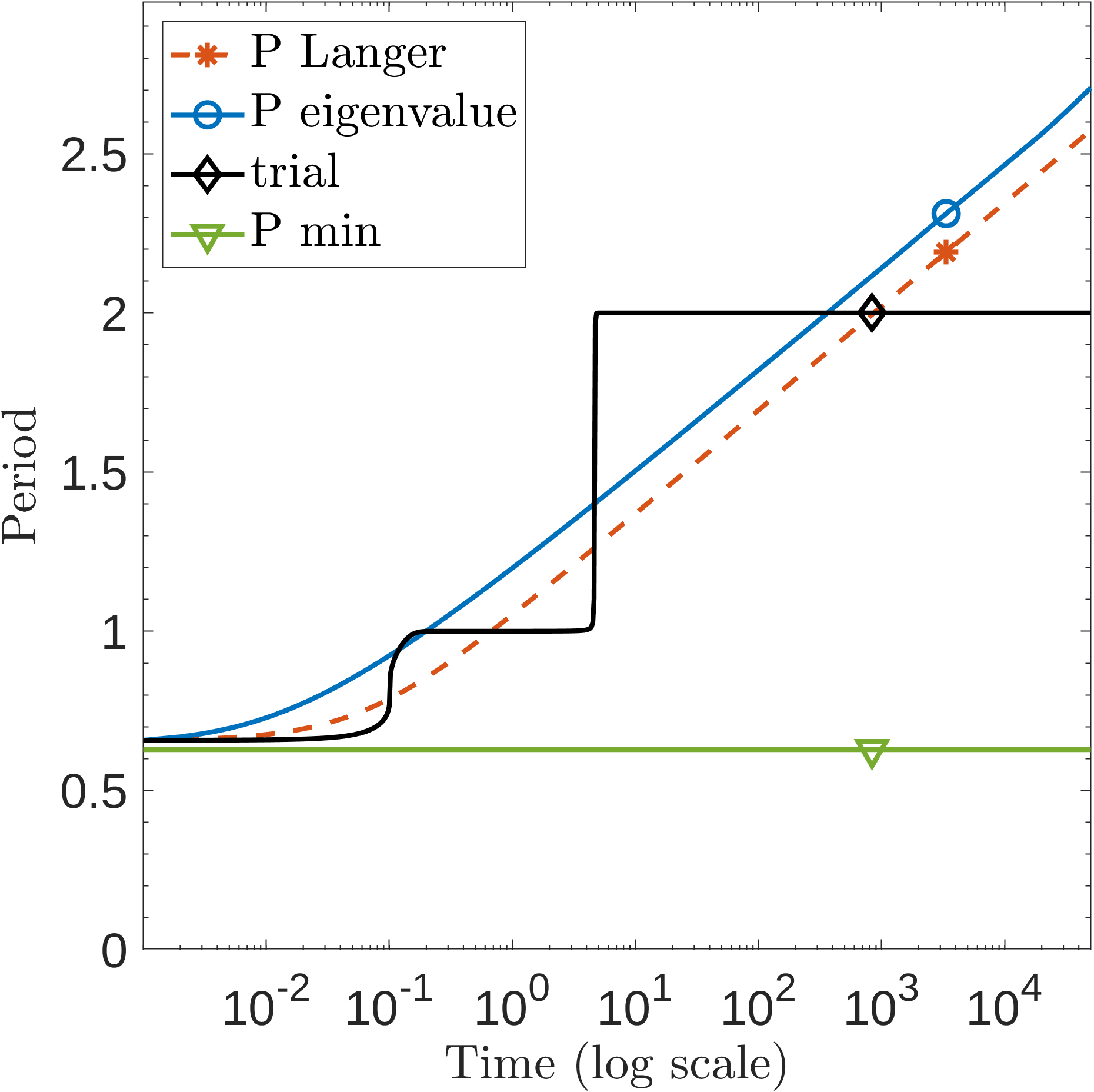}
\caption{\label{energies_av_b_small_interval_large_time}  (linear-log scale) Period of $\phi$ vs. time.}
\end{subfigure}
\caption{\label{energies_av_small_interval_large_time}
Free energy and periods of 1 trial of the Cahn-Hilliard equation ($\kappa=0.01$, $\alpha=\beta=1$, $N=2^{11}=2048$, Domain: $[-1,1]$, $\Delta t=0.001$, final time $T=50,000$), along with the maximum free energy \eqref{maximum-energy}, the kink energy \eqref{E_kink}, and the Langer and eigenvalue predictions.  As expected, we see in (a) that the eigenvalue and Langer predictions, which are based on the full-space domain, drop below the trial simulation's lowest energy state, as it is on the periodic domain, and therefore has one additional kink.
}
\end{figure}

\FloatBarrier

\section{Conclusion} \label{conclusions-section}

In this paper, we have introduced a new measure of coarseness for processes modeled by Cahn--Hilliard equations, 
and used this measure to compare two analytic descriptions of coarsening with computationally generated 
dynamics. Notably, our measure of coarseness circumvents the need for any \textit{a priori} assumptions about the 
structure of the solution, such as near-periodicity, allowing us to associate a value of coarseness with any 
solution to \eqref{ch}. 

In future work, we plan to use our measure of coarseness to investigate coarsening dynamics in systems obtained 
when Cahn--Hilliard equations are coupled with a fluids equation to create a model of two-phase flow. In particular, 
we will start by coupling the Cahn--Hilliard equations considered here with the viscous Burgers equation---creating a   
natural one-dimensional analog of the Cahn--Hilliard--Navier--Stokes system. Our expectation is that coupling will
have a substantial effect on coarsening rates, and our goal will be to meaningfully quantify this effect.  
We will also study theoretical properties of the coupled Burgers--Cahn--Hilliard equations, such as global well-posedness. 
Our long-term goal is to extend our study to multi-dimensional Cahn--Hilliard Systems, the Cahn--Hilliard--Navier--Stokes 
system, and other related equations. Although our 1D investigations point toward promising directions to pursue in 
such studies, we expect new phenomena to emerge in those cases. In particular, new analytic tools are needed for 
higher dimensional cases in order to evaluate and analyze coarsening rates.

\appendix

\section{Appendix: Proofs of Propositions \ref{period-proposition} and \ref{specific-F-proposition}} 
\label{proofs-appendix}

\begin{proof}[Proof of Proposition \ref{period-proposition}.]
In order to abbreviate notation, we will write 
\begin{equation*}
    p(a) = 2 \sqrt{2 \kappa} \tilde{p} (a),
    \quad \tilde{p} (a) = \int_0^a \frac{dy}{\sqrt{F(y) - F(a)}}. 
\end{equation*}
By setting $z = y/a$, we can express $\tilde{p} (a)$ as 
\begin{equation*}
    \tilde{p} (a) = \int_0^1 \frac{a}{\sqrt{F(az) - F(a)}} dz,
\end{equation*}
for which we can justify differentiating through the integral. This
gets us to the relation 
\begin{equation} \label{p-tilde-equation}
    \tilde{p}'(a)
    = \int_0^1 \frac{(F(az) - \frac{az}{2} F' (az)) - (F(a) - \frac{a}{2} F' (a))}{(F(az) - F(a))^{3/2}} dz.
\end{equation}
In order to establish positivity of $\tilde{p}'(a)$, we will set 
\begin{equation*}
G(y) := F(y) - \frac{y}{2} F' (y),
\end{equation*}
and show that $G(az) - G(a)$ is positive for all $z \in (0, 1)$.  
First, $G' (y) = \frac{1}{2} F'(y) - \frac{y}{2}F''(y)$, so in 
particular $G' (0) = \frac{1}{2} F'(0) = 0$. 
Next, by assumption, $G''(y) = - \frac{y}{2} F'''(y) < 0$
for all $y \in (0, a)$, so $G'(y)$ is a decreasing function,
and we must have $G' (y) < 0$ for all $y \in (0, a)$. It 
follows that $G (y)$ is a decreasing function, so $G (a) < G (az)$
for all $z \in (0, 1)$, giving the second claim.

For the first claim, it is convenient to return to $y = az$ and 
rearrange \eqref{p-tilde-equation} as 
\begin{equation} \label{p-tilde-equation2}
    \tilde{p}'(a)
    = \int_0^a \frac{\frac{1}{2} (F' (a) - F' (y)) + \frac{1}{a} (F(y) - F(a))}{(F(y) - F(a))^{3/2}} dy
    + \frac{1}{2} \lim_{\tau \to a^-} \int_0^{\tau} \frac{F'(y) - \frac{y}{a} F'(y)}{(F(y) - F(a))^{3/2}} dy,
\end{equation}
where the final integral has been expressed as a limit to justify 
integrating by parts. Integrating by parts, we find 
\begin{equation*}
    \begin{aligned}
\lim_{\tau \to a^-} \int_0^{\tau} \frac{F'(y) - \frac{y}{a} F'(y)}{(F(y) - F(a))^{3/2}} dy
&= - \int_0^a \frac{\frac{2}{a} (F(y) - F(a))}{(F(y) - F(a))^{3/2}} dy
+ \lim_{\tau \to a^-} \frac{-2 (1 - \frac{y}{a})}{\sqrt{F (y) - F(a)}} \Bigg|_0^\tau \\
&= - \int_0^a \frac{\frac{2}{a} (F(y) - F(a))}{(F(y) - F(a))^{3/2}} dy
+ \frac{2}{\sqrt{F (0) - F(a)}}. 
    \end{aligned}
\end{equation*}
Recalling the factor of $1/2$, we see that the first summand on the right-hand side of 
this last expression cancels with the second part of the first integral on the right-hand side of 
\eqref{p-tilde-equation2}, leaving 
\begin{equation*}
    \tilde{p}' (a)
    = \frac{1}{\sqrt{F (0) - F(a)}} 
    + \int_0^a \frac{\frac{1}{2} (F' (a) - F' (y))}{(F(y) - F(a))^{3/2}} dy. 
\end{equation*}
Recalling the specification $p(a) = 2 \sqrt{2 \kappa} \tilde{p} (a)$, we see 
that the proof is complete. 
\end{proof}

\begin{proof}[Proof of Proposition \ref{specific-F-proposition}.]
    Items (i) and (ii) are clear from evaluation of the indicated formulas at 
    the specific family of bulk free energy densities \eqref{quarticF}. For Item (iii),
    we first observe that for $F$ as specified in \eqref{quarticF}, we can 
    write 
    \begin{equation*}
        F(y) - F(a)
        = \left(\frac{\alpha}{4} y^2 + \frac{F(a) - F(0)}{a^2}\right) (y^2 - a^2).
    \end{equation*}
    In this case, \eqref{integralu} becomes 
    \begin{equation*} 
    \int_0^{\bar{\phi} (x;a)} 
    \frac{dy}{\sqrt{\frac{2}{\kappa} (\frac{\alpha}{4} y^2 + \frac{F(a) - F(0)}{a^2}) (y^2 - a^2)}} = x.
    \end{equation*}  
    In order to express this as a Jacobi elliptic function, we first multiply both 
    sides by the factor $(1/a) \sqrt{- 2 (F(a) - F(0))/\kappa}$ to obtain 
    \begin{equation*} 
    \int_0^{\bar{\phi} (x;a)} 
    \frac{\frac{1}{a} \sqrt{\frac{-2 (F(a) - F(0))}{\kappa}}}
    {\sqrt{\frac{2}{\kappa} (\frac{\alpha}{4} y^2 + \frac{F(a) - F(0)}{a^2}) (y^2 - a^2)}} dy 
    = \frac{x}{a} \sqrt{\frac{-2 (F(a) - F(0))}{\kappa}},
    \end{equation*}  
    which we can rearrange to 
    \begin{equation*} 
    \int_0^{\bar{\phi} (x;a)} 
    \frac{dy}{\sqrt{(a^2 - y^2) (1 + \frac{\alpha a^2}{4(F(a) - F(0))}y^2)}}
    = \frac{x}{a} \sqrt{\frac{-2 (F(a) - F(0))}{\kappa}}.
    \end{equation*}  

    If we now set $y = a \sin \theta$ and specify $k$ as 
    in \eqref{kdefined}, we obtain the relation 
    \begin{equation*}
        \int_0^{\bar{\theta}(x;a)} \frac{d \theta}{\sqrt{1 - k^2 \sin^2 \theta}}
        = \frac{x}{a} \sqrt{\frac{- 2 (F(a) - F(0))}{\kappa}},
    \end{equation*}
    where $\sin \bar{\theta}  = \bar{\phi}/a$. By definition of the Jacobi 
    elliptic function $\operatorname{sn}$ (see Remark \ref{periodic-solution-details}), 
    \begin{equation*}
    \operatorname{sn} \left(\sqrt{\frac{- 2 (F(a) - F(0))}{\kappa}} \frac{x}{a}, k\right)
    = \sin \bar{\theta} (x; a) = \frac{\bar{\phi} (x; a)}{a},
    \end{equation*}
    from which Item (iii) follows immediately.  

    Turning to Item (iv), it is straightforward to verify directly 
    that \eqref{kink-solution} is indeed a stationary solution 
    to \eqref{ch} with $F$ as in \eqref{quarticF}. For the 
    energy claim, we need to evaluate $E$ in \eqref{ch_energy} 
    at $K(x)$. For this, we first observe that $K (x)$ satisfies
    the equation 
    \begin{equation*}
        K' = \sqrt{\frac{2}{\kappa} F (K(x))},
    \end{equation*}
    so that 
    \begin{equation*}
        E(K(x)) = \int_{-L}^{+L} F (K (x)) + \frac{\kappa}{2} \cdot \frac{2}{\kappa} F (K (x)) dx
        = 2 \int_{-L}^{+L} F(K(x)) dx. 
    \end{equation*}
    We make the change of variables $y = K(x)$, so that 
    \begin{equation*}
        dy = K'(x) dx = \sqrt{\frac{2}{\kappa} F(K(x))} dx,
    \end{equation*}
    giving 
    \begin{equation*}
    \begin{aligned}
        E(K(x))
        &= \sqrt{2 \kappa} \int_{K (-L)}^{K (+L)} \sqrt{F(y)} dy
        = \sqrt{\frac{\kappa \alpha}{2}} \int_{K (-L)}^{K (+L)} \frac{\beta}{\alpha} - y^2 dy \\
        &= \sqrt{2 \kappa \alpha} K (L) \Big(\frac{\beta}{\alpha} - \frac{K (L)^2}{3} \Big),
    \end{aligned}
    \end{equation*}
    where in obtaining the final equality, we integrated directly 
    and observed the relation $K (-L) = -K(L)$. Finally, the 
    statement about $E_{\min}^{\infty}$ is clear from the limit 
    \begin{equation*}
        \lim_{L \to \infty} K(L) = \sqrt{\frac{\beta}{\alpha}}. 
    \end{equation*}
\end{proof}

\section*{Acknowledgments}
  The authors would like to thank editors and the reviewers for their insightful and constructive comments, which have helped us to improve significantly the quality of the paper.
  A.L. was supported in part by NSF grants DMS-2206762, DMS-2510494, and CMMI-1953346, and USGS Grant No. G23AC00156-01. Q.L. was partially supported by the Simons Foundation (SFI-MPS-TSM-00013384).

\bibliographystyle{abbrv}
\bibliography{refs}
\end{document}